\documentclass[onefignum, onetabnum]{siamart250211}

\usepackage{lipsum}
\usepackage{amsfonts}
\usepackage{graphicx}
\usepackage{epstopdf}
\usepackage{algorithmic}
\usepackage{booktabs}
\usepackage{float}
\ifpdf
  \DeclareGraphicsExtensions{.eps,.pdf,.png,.jpg}
\else
  \DeclareGraphicsExtensions{.eps}
\fi

\newsiamremark{remark}{Remark}
\newsiamremark{hypothesis}{Hypothesis}
\newsiamremark{example}{Example}
\crefname{hypothesis}{Hypothesis}{Hypotheses}
\newsiamthm{claim}{Claim}

\headers{High-Order Meshfree Surface Integration}{D. R. Venn and S. J. Ruuth}

\title{High-Order Meshfree Surface Integration, Including Singular Integrands\thanks{
\funding{We acknowledge the support of the Natural Sciences and
Engineering Research Council of Canada (NSERC), [funding reference numbers 579365-2023 (DV), RGPIN-2022-03302 (SR)].}}}

\author{Daniel R. Venn\thanks{Department of Mathematics, Simon Fraser University, Burnaby, BC 
  (\email{dvenn@sfu.ca}).}
\and Steven J. Ruuth\thanks{Department of Mathematics, Simon Fraser University, Burnaby, BC 
  (\email{sruuth@sfu.ca}).}}

\usepackage{amsopn}

\makeatletter
\newcommand*{\addFileDependency}[1]{%
  \typeout{(#1)}%
  \@addtofilelist{#1}%
  \IfFileExists{#1}{}{\typeout{No file #1.}}%
}
\makeatother

\ifpdf
\hypersetup{
  pdftitle={High-Order Meshfree Surfaces Integration, Including Singular Integrands},
  pdfauthor={D. R. Venn and S. R. Ruuth}
}
\fi

\global\long\def\b#1{\left(#1\right)}%

\global\long\def\abs#1{\left|#1\right|}%

\global\long\def\d{\text{d}}%

\global\long\def\eval{\biggr|}%

\global\long\def\sb#1{\left[#1\right]}%

\global\long\def\cb#1{\left\{  #1\right\}  }%

\global\long\def\lc{\varepsilon}%

\global\long\def\R{\mathbb{R}}%

\global\long\def\C{\mathbb{C}}%

\global\long\def\code#1{\mathtt{#1}}%

\global\long\def\O{\mathcal{O}}%

\global\long\def\ss{\subseteq}%

\global\long\def\norm#1{\left\Vert #1\right\Vert }%

\global\long\def\vp{\varphi}%

\global\long\def\F{\mathcal{F}}%

\global\long\def\C{\mathbb{C}}%

\global\long\def\H{\mathcal{H}}%

\global\long\def\E{\mathcal{E}}%

\global\long\def\L{\mathcal{L}}%

\global\long\def\G{\mathcal{G}}%

\global\long\def\vc#1{\boldsymbol{#1}}%

\global\long\def\x{\vc x}%

\global\long\def\y{\vc y}%

\begin{document}

\maketitle

\begin{abstract}
 We develop and test high-order methods for integration on surface
point clouds. The task of integrating a function on a surface arises
in a range of applications in engineering and the sciences, particularly
those involving various integral methods for partial differential
equations. Mesh-based methods require a curved mesh for high-order
convergence, which can be difficult to reliably obtain on many surfaces,
and most meshfree methods require the ability to integrate a set of
functions (such as radial basis functions) exactly on the domain of
interest; these integrals are generally not known in closed form on
most surfaces. We describe two methods for integrating on arbitrary,
piecewise-smooth surfaces with or without boundary. Our approaches
do not require a particular arrangement of points or an initial triangulation
of the surface, making them completely meshfree. We also show how
the methods can be extended to handle singular integrals while maintaining
high accuracy without changing the point density near singularities.

\end{abstract}


\section{Introduction}
A classical problem in a wide range of scientific disciplines is estimating
$\int_{S}f$ from data obtained in a domain $S$. There are a variety
of approaches, including solving moment fitting equations \cite{thiag14,thiag16},
interpolating data with radial basis functions (RBFs) \cite{shu03,glaub23},
integrating polynomials on polyhedra (often for use with finite elements)
\cite{holdy08,mousa10}, and Green's theorem-based methods for dimension
reduction \cite{gunde21}. The majority of these methods (except for
some RBF methods) typically determine quadrature points and weights
for some predetermined class of domains such as polyhedra or domains
bounded by algebraic curves. In cases where $f$ is not known everywhere,
evaluating $f$ at quadrature points requires another interpolation
step. Most methods (including many RBF methods) also require knowledge
of the moments of a class of functions on the domain of interest.
These moments are generally unknown, except on the simplest domains.

Surfaces present a particular challenge for integration as they are
often more difficult to mesh. For meshed methods, it should be noted
that simply increasing the degree of the polynomial used to interpolate
a function on a standard surface mesh with flat, polygonal elements
will not improve the order of convergence beyond second order; the
surface itself must also be approximated. This can be done with curved
elements, such as those used in high-order surface finite element methods
\cite{demlo09}, or by using some other approximation of the surface,
such as a least-squares approach combined with standard triangular
elements as in \cite{ray12}. We note that actually producing a high-order
surface mesh can be challenging and time-consuming, and often involves
interpolating a low-order mesh, which may require using a meshfree
method (see \cite{zhao19}, for example).

Another possible solution is to use a standard quadrature method in
the embedding space and discretize surface integrals using a regularization
of the Dirac delta. Such approaches can be used with level-set methods
(see \cite{engqu05}, for example) or when a closest-point representation
is available \cite{kubli16}, and can be made high order with a suitable
choice of regularization \cite{waldn99}. However, these techniques
require increasingly expensive discretizations of a neighbourhood
of the surface in the embedding space as the accuracy of the regularization
is improved. Furthermore, an approximation of the signed distance
function for the surface is needed, which may involve additional computation.
These methods are also generally only suitable for differentiable
closed surfaces for which a level set description is known, excluding
surfaces with boundary. Moreover, for surfaces that are only piecewise-smooth,
convergence is at a reduced rate.

For meshfree methods, techniques that rely on knowing the exact integral
of a class of functions (such as RBFs) on the domain are typically
not available, since these exact integrals are often unknown. Meshfree
integration on surface point clouds is therefore often limited to
Monte Carlo methods, which have been employed for eigenfunction estimation
using the weak form of differential operators \cite{harli23,yan23}.
Even though the underlying meshfree interpolation technique used is
high-order, Monte Carlo integration is not, which limits the overall
rate of convergence. Monte Carlo integration also requires knowledge
of the distribution from which a point cloud is sampled, which may
not be available, and it may not be straightforward to sample points
from a surface uniformly. If the points satisfy a minimum distance
condition, so that no two points are too close together, a potential
integration technique for non-uniform point clouds is to use local
projections to the tangent plane and a triangulation to compute Voronoi
cell areas (see \cite{luo09}), though this has a limited order of
convergence. Existing high-order meshfree integration methods on
surfaces are often restricted to specific surfaces where exact integrals
of RBFs are known, such as the sphere \cite{kinde22} or an explicit basis for a Hilbert space on the surface can be written down, as in Stein reproducing kernel methods from the statistics literature (see \cite{barp22}, for example). Reeger, Fornberg,
and Watts \cite{reege16} have also introduced a high-order ($\O\b{h^{7}}$,
where $h$ is the point spacing), RBF-based method for general closed
surfaces that requires first obtaining a triangulation of the point
cloud. This is essentially a hybrid approach between a meshfree and
meshed method in that while a triangulation of the point cloud is
required, a meshfree method is ultimately used to compute integrals
on each surface patch projected to a triangle with high accuracy.
A version for surfaces with boundary has also been developed \cite{reege18}.

Our contribution consists of two super-algebraically convergent methods
based on the divergence theorem and norm-minimizing Hermite-Birkhoff interpolation
schemes. Such interpolation schemes include Hermite RBFs \cite{sun94,frank98}
and minimum norm methods using Fourier series \cite{chand18,venn}.
We review these interpolation methods in Section \ref{sec:Background},
including a generalized convergence statement in Proposition \ref{prop:Let--be},
a discussion on a general approach to high-order convergence in \ref{subsec:High-Order-Bounds-on},
and two examples in \ref{subsec:Hilbert-Spaces-for}. We then introduce
our integration methods in Section \ref{sec:Methods-and-Analysis}.
Our approaches do not require specific quadrature points (the points
may be randomly scattered and non-uniformly selected) and only require
the knowledge of basic domain information such as normal vectors at
scattered points on the surface. Method 1 can be used to directly evaluate
the ratio of two integrals on a closed surface or a surface with boundary; this is useful when a constant multiple of the integral
is all that is needed. For example, the method could be used when
the average value of a function on the surface is desired. Method 2 requires surfaces to have boundary, but we discuss simple strategies
for dividing closed surfaces into multiple surfaces with boundary
in Subsection \ref{subsec:Meshfree-Domain-Decomposition}. Unlike
the first method, method 2 estimates integrals themselves
rather than a ratio, which can be used to estimate surface area with
high precision. We show that these schemes converge at the same rate
as the underlying interpolation methods, regardless of the point distribution;
this leads to arbitrarily high-order schemes. Then, we test the integration
techniques for a variety of problems on flat domains and on surfaces
in Section \ref{sec:Numerical-Tests}. We conclude by introducing
a more general view of norm-minimizing meshfree methods suitable for
handling singularities, and then adapt our methods for singular integrals
in Subsection \ref{subsec:High-Order-Approximation-of}; such integrals
appear in boundary integral solutions to PDEs.
Code for running the numerical tests in this paper is available at \href{https://github.com/venndaniel/meshfree_integration}{https://github.com/venndaniel/meshfree\_integration}.

\section{Background\label{sec:Background}}

We start by giving a brief introduction to solving partial differential
equations (PDEs) with meshfree methods. In particular, we look at
symmetric, strong-form collocation approaches that choose a Hermite-Birkhoff
interpolant that minimizes a norm in a Hilbert space. These approaches
can be viewed as constrained optimization problems, which is a vital
point for understanding Method 1 in Subsection \ref{subsec:Method-=0000231:-Poisson}.

\subsection{Norm-Minimizing Hermite-Birkhoff Interpolants\label{subsec:Norm-Minimizing-Hermite-Birkhoff}}

To solve a PDE, one possible approach is strong-form collocation.
In such methods, we search for an approximate solution $\tilde{u}$
that satisfies the PDE on a discrete set of points $\cb{\x_{j}}_{j=1}^{\tilde{N}}$
contained in a domain $S\subset\R^m$. More precisely, if we consider the problem
\begin{align}
\F u & =f\text{ on }S,\label{eq:pde1}
\end{align}
where $\F$ is a linear differential operator on $S$ and $f:S\to\C$
is a function, we would search for an approximate solution $\tilde{u}$
such that
\begin{equation}
\b{\F\tilde{u}}\b{\x_{j}}=f\b{\x_{j}}\text{ for each } j\in\cb{1,2,\ldots,\tilde{N}}.\label{eq:hbinterp}
\end{equation}

Boundary conditions can be handled in a number of ways. A standard
approach would be to select $\tilde{u}$ from a finite-dimensional
space of functions that match the boundary condition of interest.
For example, a cosine series could be used for a 1D problem with homogeneous
Neumann boundary conditions. However, we are interested in irregular
domains and surfaces, where suitable function spaces are generally
not known. One possible approach is to simply include the boundary
conditions as additional interpolation conditions. That is, if we
have the boundary condition
\begin{equation}
\G u=g\text{ on }\partial S,\label{eq:pdebound}
\end{equation}
we further require that in addition to satisfying Eq. (\ref{eq:hbinterp}),
our approximate solution $\tilde{u}$ satisfies
\begin{equation}
\b{\G\tilde{u}}\b{\y_{j}}=g\b{\y_{j}}\text{ for each }j\in\cb{1,2,\ldots,\tilde{N}_{\partial}},\label{eq:hbinterp boundary}
\end{equation}
where $\{\y_{j}\}_{j=1}^{\tilde{N}_{\partial}}\subset\partial S$
is a set of points on the boundary of $S$. The goal is then to ensure
that $\F\tilde{u}\to f$ on $S$ and $\G\tilde{u}\to g$ on $\partial S$
in some sense as $\tilde{N},\tilde{N}_{\partial}\to\infty$. If the
PDE given by (\ref{eq:pde1}) and (\ref{eq:pdebound}) is suitably
well-posed, we may then be able to show that $\tilde{u}\to u$, with
the manner of convergence ($L^{2},L^{\infty},H^{p}$, pointwise, etc.)
depending on the PDE and method chosen (see, for example, Proposition 11 of \cite{venn}). 

Of course, simply imposing $\b{\F\tilde{u}}\b{\x_{j}}=f\b{\x_{j}}$
on each $\x_{j}$ is not enough to ensure that $\F\tilde{u}\to f$
as more collocation points are added. A classic counterexample is
the interpolation of Runge's function ($f\,(x) =1/(1+25x^{2})$) with
polynomials of increasing degree on $\sb{-1,1}$ and evenly spaced
points $\cb{\x_{j}}$; this corresponds to an identity operator $\F$
and results in an interpolant $\tilde{u}$ that is unbounded as $\tilde{N}\to\infty$.
While basic, this counterexample demonstrates one possible manner
in which interpolation can fail: the interpolant can become unbounded.
This motivates a constrained approach. For some Hilbert space of functions $\H$ defined on a domain $\Omega \supseteq S$, we solve
\begin{align}
\text{minimize, over }\tilde{v}\in\H: & \norm{\tilde{v}}_{\H},\label{eq:optprob}\\
\text{subject to } & \b{\F\tilde{v}}\b{\x_{j}}=f\b{\x_{j}}\text{ for each \ensuremath{j}\ensuremath{\in\cb{1,2,\ldots,\tilde{N}},}}\nonumber \\
 & \b{\G\tilde{v}}\b{\y_{j}}=g\b{\y_{j}}\text{ for each \ensuremath{j}\ensuremath{\in\cb{1,2,\ldots,\tilde{N}_{\partial}}.}}\nonumber 
\end{align}
It is convenient to define a linear operator $\L:\H\to\C^{\tilde{N}+\tilde{N}_{\partial}}$
and some $\vc f\in\C^{\tilde{N}+\tilde{N}_{\partial}}$ such that
this problem becomes

\begin{align}
\text{minimize, over }\tilde{v}\in\H: & \norm{\tilde{v}}_{\H},\label{eq:optproblem}\\
\text{subject to } & \L\tilde{v}=\vc f.\nonumber 
\end{align}

For this problem to be uniquely solvable, we require $\L$ to be bounded
in the $\H\to\C^{\tilde{N}+\tilde{N}_{\partial}}$ sense, which means
that both $\F$ and $\G$ evaluated at individual points must be bounded
as operators from $\H$ to $\C$. We also need the constraint set
to be non-empty for each $\vc f$; $\L$ must be surjective. We note that for reasonable choices of $\H$ (examples are discussed in Subsection \ref{subsec:Hilbert-Spaces-for}), $\L$  is generally surjective for a finite number of collocation points even if the PDE given by (\ref{eq:pde1}) and
(\ref{eq:pdebound}) is not solvable globally. In this case, problem
(\ref{eq:optproblem}) has a few important properties. Let $\tilde{u}\in\H$
be the solution to (\ref{eq:optproblem}), then $\tilde{u}\in\mathcal{R}\,(\L^{*})=(\mathcal{N}\,(\L))^{\perp}$
(noting that $\mathcal{R}\b{\L^{*}}$ is finite-dimensional and therefore
closed). We can therefore find $\tilde{u}$ by solving
\begin{align}
\L\L^{*}\vc{\beta} & =\vc f,\text{\ensuremath{\quad}}\tilde{u}=\L^{*}\vc{\beta}.\label{eq:kernel}
\end{align}

The linear problem given by (\ref{eq:kernel}) is finite-dimensional,
and $\L\L^{*}$ is symmetric and positive-definite as long as the
constraint set is non-empty; $\L\L^{*}$ corresponds to the interpolation
matrix $\vc{\Phi}$ for (Hermite) RBF methods. This gives us the
term “symmetric meshfree methods” to refer to methods that arise
from the solution of (\ref{eq:optproblem}), since $\L\L^{*}$ is
symmetric even when the operator $\F$ is not. Note that there are
also non-symmetric meshfree methods, such as Kansa's \cite{kansa90},
for which the methods and results of this paper do not apply.

Now, as long as $\tilde{u}$ remains bounded as $\tilde{N},\tilde{N}_{\partial}\to\infty$,
$\tilde{u}$ converges in $\H$. This is due to the following result
in functional analysis. We stated versions of this result specific
to meshfree methods with affine constraints in our previous works
\cite{venn,venn24b}. The version presented here generalizes these
results to convex constraint sets in Hilbert spaces.
\begin{proposition}
\label{prop:Let--be}Let $\cb{V_{n}}_{n=1}^{\infty}$ be a collection
of closed, non-empty, convex sets in a Hilbert space $\H$ such that
$V_{1}\supseteq V_{2}\supseteq\ldots$ and define
\[
\tilde{u}_{n}=\text{argmin}\cb{\norm{\tilde{u}}_{\H}:\tilde{u}\in V_{n}}.
\]
Assume $\cb{\norm{\tilde{u}_{n}}_{\H}}_{n=1}^{\infty}$ is bounded
by a constant $B>0$, then $\tilde{u}_{n}$ converges in $\H$ to
some $\tilde{u}_{\infty}\in\bigcap_{n=1}^{\infty}V_{n}$.
\end{proposition}

\begin{proof}
Note that $\tilde u_{n}$ is well-defined due to a standard result in functional
analysis (see Theorem 3.3-1 of Kreyszig's text \cite{kreys91}, for
example), since $V_{n}$ is non-empty, closed, and convex by assumption.
Furthermore, $\cb{\norm{\tilde{u}_{n}}_{\H}}_{n=1}^{\infty}$ is non-decreasing,
since $V_{n_{1}}\supseteq V_{n_{2}}$ whenever $n_{1}<n_{2}$. For
any convex combination $\b{1-\alpha}\tilde{u}_{n_{1}}+\alpha\tilde{u}_{n_{2}}$
such that $n_{1}<n_{2}$ and $\alpha\in\b{0,1}$,
\begin{align}
\norm{\tilde u_{n_{1}}}_{\H}^{2} & \le\norm{\alpha\b{\tilde{u}_{n_{2}}-\tilde{u}_{n_{1}}}+\tilde{u}_{n_{1}}}_{\H}^{2}\nonumber \\
 & =\norm{\tilde u_{n_{1}}}_{\H}^{2}+\alpha^{2}\norm{\tilde{u}_{n_{2}}-\tilde{u}_{n_{1}}}_{\H}^{2}+2\alpha\,\text{Re}\b{\tilde{u}_{n_{2}}-\tilde{u}_{n_{1}},\tilde{u}_{n_{1}}}_{\H}\nonumber \\
\implies0 & \le\alpha\norm{\tilde{u}_{n_{2}}-\tilde{u}_{n_{1}}}_{\H}^{2}+2\,\text{Re}\b{\tilde{u}_{n_{2}}-\tilde{u}_{n_{1}},\tilde{u}_{n_{1}}}_{\H}.\label{eq:zeroless}
\end{align}
Note that the first line above is because $\tilde{u}_{n_{2}}\in V_{n_{2}}\ss V_{n_{1}}$,
$\alpha\tilde{u}_{n_{1}}+\b{1-\alpha}\tilde{u}_{n_{2}}\in V_{n_{1}}$
by convexity, and $\tilde{u}_{n_{1}}=\text{argmin}\cb{\norm{\tilde{u}}_{\H}:\tilde{u}\in V_{n_{1}}}$.
Now, since (\ref{eq:zeroless}) holds for each $\alpha\in\b{0,1}$,
\begin{align}
\text{Re}\b{\tilde{u}_{n_{2}}-\tilde{u}_{n_{1}},\tilde{u}_{n_{1}}}_{\H} & \ge0\label{eq:sign}\\
\implies\text{Re}\b{\tilde{u}_{n_{2}},\tilde{u}_{n_{1}}}_{\H} & \ge\norm{\tilde{u}_{n_{1}}}_{\H}^{2}\label{eq:diff}
\end{align}
Then,
\begin{align}
\norm{\tilde{u}_{n_{2}}-\tilde{u}_{n_{1}}}_{\H}^{2} & =\text{Re}\b{\tilde{u}_{n_{2}},\tilde{u}_{n_{2}}-\tilde{u}_{n_{1}}}-\text{Re}\b{\tilde{u}_{n_{1}},\tilde{u}_{n_{2}}-\tilde{u}_{n_{1}}}\nonumber \\
 & \le\norm{\tilde{u}_{n_{2}}}_{\H}^{2}-\norm{\tilde{u}_{n_{1}}}_{\H}^{2}\text{, due to (\ref{eq:sign}) and }(\ref{eq:diff})\nonumber \\
 & \le2B\b{\norm{\tilde{u}_{n_{2}}}_{\H}-\norm{\tilde{u}_{n_{1}}}_{\H}}\text{, since \ensuremath{\norm{\tilde{u}_{n}}_{\H}\le B} for each \ensuremath{n}}.\label{eq:bound1}
\end{align}

Finally, $\cb{\norm{\tilde{u}_{n}}_{\H}}_{n=1}^{\infty}$ is bounded
and non-decreasing, and therefore converges, which implies $\norm{\tilde{u}_{n_{2}}-\tilde{u}_{n_{1}}}_{\H}\to0$
as $n_{1},n_{2}\to\infty$ due to (\ref{eq:bound1}). The sequence
$\cb{\tilde{u}_{n}}_{n=1}^{\infty}$ is therefore Cauchy and must
converge to some $\tilde{u}_{\infty}\in\H$, since $\H$ is a Hilbert
space. Furthermore, for each $m$, $\cb{\tilde{u}_{n}}_{n=m}^{\infty}\subset V_{m}$,
which is closed, so $\tilde{u}_{\infty}\in V_{m}$ as well. Therefore,
$\tilde{u}_{\infty}\in\bigcap_{n=1}^{\infty}V_{n}$. 
\end{proof}

We note that the constraint set for problem (\ref{eq:optproblem})
is closed as long as $\L$ is bounded, which allows us to apply Proposition
\ref{prop:Let--be}.  Also note the following corollary of Proposition
\ref{prop:Let--be}.
\begin{corollary}
\label{cor:Where--is}Where $\tilde{u}_{n}$ and $\tilde{u}_{\infty}$
are defined as in Proposition \ref{prop:Let--be} and the assumptions
of Proposition \ref{prop:Let--be} hold,
\[
\tilde{u}_{\infty}=\text{argmin}\cb{\norm{\tilde{u}}_{\H}:\tilde{u}\in\bigcap_{n=1}^{\infty}V_{n}}.
\]
\end{corollary}

\begin{proof}
Let $\tilde{u}_{\text{min}}=\text{argmin}\cb{\norm{\tilde{u}}_{\H}:\tilde{u}\in\bigcap_{n=1}^{\infty}V_{n}}$;
note $\bigcap_{n=1}^{\infty}V_{n}$ is non-empty since we already
know $\tilde{u}_{\infty}$ exists and is in $\bigcap_{n=1}^{\infty}V_{n}$.
Then, $\norm{\tilde{u}_{n}}_{\H}\to\norm{\tilde{u}_{\infty}}_{\H}$
and $\norm{\tilde{u}_{\infty}}_{\H}\ge\norm{\tilde{u}_{\text{min}}}_{\H}\ge\norm{\tilde{u}_{n}}_{\H}$
for each $n$, so $\norm{\tilde{u}_{\infty}}_{\H}=\norm{\tilde{u}_{\text{min}}}_{\H}$.
Finally, $\tilde{u}_{\infty}=\tilde{u}_{\text{min}}$ because $\bigcap_{n=1}^{\infty}V_{n}$
is convex and therefore its minimum-norm element is unique.
\end{proof}
For our optimization problem applied to PDEs, assume that $\{\x_{j}\}_{j=1}^{\infty}$
is dense in $S$ and that $\{\y_{j}\}_{j=1}^{\infty}$ is dense in
$\partial S$. Then, Corollary \ref{cor:Where--is} implies that if
our solution $\tilde{u}_{\tilde{N},\tilde{N}_{\partial}}$ to the
optimization problem (\ref{eq:optproblem}) remains bounded in $\H$
as $\tilde{N},\tilde{N}_{\partial}\to\infty$, there is at least one
solution to the PDE on a dense set of points in $S$ and $\partial S$.
If $\F u,\G u$ are continuous for functions $u\in\H$, there is a
solution to the PDE on all of $S$ and $\partial S$, and our solutions
to (\ref{eq:optproblem}) converge to $\tilde{u}_{\infty}$ where
\[
\tilde{u}_{\infty}=\text{argmin}\cb{\norm{\tilde{u}}_{\H}:\F\tilde{u}\eval_{S}=f\text{ and }\G\tilde{u}\eval_{\partial S}=g}.
\]
That is, $\tilde{u}_{\tilde{N},\tilde{N}_{\partial}}$ converges to
the norm-minimizing solution to the PDE given by (\ref{eq:pde1})
and (\ref{eq:pdebound}). Also note that if at least one strong-form
solution to the PDE of (\ref{eq:pde1}) and (\ref{eq:pdebound}) exists
and can be extended to a function $u\in\H$, then $u$ is feasible
for each (\ref{eq:optproblem}) and $\Vert\tilde{u}_{\tilde{N},\tilde{N}_{\partial}}\Vert_{\H}\le\norm u_{\H}$
for each $n$, so Proposition \ref{prop:Let--be} and Corollary \ref{cor:Where--is}
apply; $\tilde{u}_{\tilde{N},\tilde{N}_{\partial}}\to\tilde{u}_{\infty}$,
where $\tilde{u}_{\infty}$ also solves the PDE of (\ref{eq:pde1})
and (\ref{eq:pdebound}).

\subsection{High-Order Bounds on Functions with Scattered Zeros\label{subsec:High-Order-Bounds-on}}

Before discussing convergence rates, we must recall a couple of definitions. 
\begin{definition}
The fill distance $h_{\text{max}}$ of a set of points $\cb{\x_{j}}_{j=1}^{\tilde{N}}$
in a domain $S$ is given by
\[
h_{\text{max}}=\sup_{\x\in S}\min_{j\in\cb{1,2,\ldots,\tilde{N}}}\norm{\x-\x_{j}}_{2}.
\]
\end{definition}

A beneficial aspect of meshfree methods is that they typically converge
in a high-order manner with respect to the fill distance. Various
RBF interpolation schemes are known to converge super-algebraically
when used to approximate functions both in and out of the native Hilbert
space $\H$ on certain types of domains (see, for example, Subsection
3.2 of \cite{narco05} for functions outside the native space). A
common requirement is that the domain satisfies an interior cone condition,
which we now define.
\begin{definition}
(See Def. 3.6 from \cite{wendl04}) A domain $U\subset\R^{m}$ satisfies
an interior cone condition if there exists an angle $\theta\in(0,\frac{\pi}{2})$
and a radius $r>0$ such that for all $\x\in U$, there exists a unit
vector $\vc{\xi}\b{\x}$ such that
\[
\cb{\x+\lambda\y:\y\in\R^{m},\norm{\y}_{2}=1,\y\cdot\vc{\xi}\b{\x}\ge\cos\theta,\lambda\in\sb{0,r}}\subset U.
\]
\end{definition}

Bounded Lipschitz domains are an example of domains that satisfy an
interior cone condition. One path towards efficiently proving convergence
of certain meshfree interpolation schemes is to note that functions
with scattered zeros satisfy high-order bounds. This can be proven
directly, as in \cite{narco05}, or using existing results from other
interpolation schemes, which is the path we have previously taken
in \cite{venn} using results from \cite{wendl04}, summarized in
the next result.
\begin{theorem}
\label{thm:(Adapted-from-Corollary}(Adapted from Corollary 6 of \cite{venn})
Let $f\in H^{\frac{m}{2}+q+\frac{1}{2}}\, (U)$ where $U\subset\R^{m}$
is bounded and satisfies an interior cone condition, and $m\ge3$
if $q=0$. Assume $f$ is exactly zero on $\{x_{k}\}_{k=1}^{\tilde{N}}$
with fill distance $h_{\text{max}}$ on $U$, then for all $p\le q$,
there exist constants $C_{m,p},h_{0,m,p}>0$ such that as long as
$h_{\text{max}}<h_{0,m,p}$,
\[
\norm f_{L^{\infty}\b U}\le C_{m,p}h_{\text{max}}^{p+\frac{1}{2}}\norm f_{H^{\frac{m}{2}+q+\frac{1}{2}}\b U}.
\]
\end{theorem}

On compact manifolds $S$ with dimension $m$, the same theorem holds,
with the sufficient conditions that $S$ can be parametrized by a
finite atlas $\cb{\sigma_{k}}_{k=1}^{M}$ with $0<M<\infty$ such
that $\sigma_{k}:U_{k}\to S$ is $C^{\left\lceil \frac{m}{2}+q+\frac{1}{2}\right\rceil }$
with bounded derivatives and has a bounded inverse tensor, and each
$U_{k}$ is bounded and satisfies an interior cone condition. The
theorem may then be applied to $f\circ\sigma_{k}$ on $U_{k}$. The
utility of Theorem \ref{thm:(Adapted-from-Corollary} is that if our
Hilbert space $\H$ is such that $\Vert\F v\Vert_{H^{\frac{m}{2}+q+\frac{1}{2}}\b S}\le A\norm v_{\H}$
and $\Vert\G v\Vert_{H^{\frac{m}{2}+q+\frac{1}{2}}\b{\partial S}}\le B\norm v_{\H}$
for each $v\in\H$ and some constants $A,B$, then we can immediately
state estimates for $\Vert\F\tilde{u}-f\Vert_{L^{\infty}\b S}$ and $\norm{\G\tilde{u}-g}_{L^{\infty}\b{\partial S}}$
where $\tilde{u}$ is the solution to (\ref{eq:optprob}). 

In particular,
if any strong-form solution $u$ to the PDE given by (\ref{eq:pde1}) and (\ref{eq:pdebound})
exists and can be extended to $\H$, and $\norm{\F v}_{H^{\frac{m}{2}+q+\frac{1}{2}}\b S}\le A\norm v_{\H}$
for each $v\in\H$, then
\begin{align*}
\norm{\F\tilde{u}-f}_{H^{\frac{m}{2}+q+\frac{1}{2}}\b S} & =\norm{\F\b{\tilde{u}-u}}_{H^{\frac{m}{2}+q+\frac{1}{2}}\b S}\le A\norm{\tilde{u}-u}_{\H}.
\end{align*}
Now, if such a strong-form solution $u$ exists, then solutions to (\ref{eq:optprob}) are bounded, $\tilde u_\infty$ exists by Proposition \ref{prop:Let--be}, and $\tilde{u}\to\tilde u_\infty$. If the sets of points $\cb{\x_{j}}_{j=1}^{\infty}$ and $\cb{\y_{j}}_{j=1}^{\infty}$ are dense on $S$ and $\partial S$ and functions in $\H$ are sufficiently smooth, then $\tilde u_\infty$ restricted to $S$ is also a strong form solution and 
\begin{align*}
\norm{\F\tilde{u}-f}_{H^{\frac{m}{2}+q+\frac{1}{2}}\b S} &\le A\norm{\tilde{u}-\tilde u_\infty}_{\H}\to 0.
\end{align*}
Therefore, if we invoke Theorem \ref{thm:(Adapted-from-Corollary} (assuming
$S$ satisfies sufficient conditions), noting that $\F\tilde{u}-f$
has scattered zeros, we have, for $p\le q$,
\[
\norm{\F\tilde{u}-f}_{L^{\infty}\b S}=o\b{h_{\text{max}}^{p+\frac{1}{2}}}.
\]
There is then a similar statement for $\norm{\G\tilde{u}-g}_{L^{\infty}\b{\partial S}}$.
Whether or not a solution $u\in\H$ to the PDE (\ref{eq:pde1}) and (\ref{eq:pdebound})
exists, Theorem \ref{thm:(Adapted-from-Corollary} still implies
\begin{align*}
\norm{\F\tilde{u}-f}_{L^{\infty}\b S} & \le C_{m,p}h_{\text{max}}^{p+\frac{1}{2}}\norm{\F\tilde{u}-f}_{H^{\frac{m}{2}+q+\frac{1}{2}}\b S}\\
 & \le C_{m,p}h_{\text{max}}^{p+\frac{1}{2}}\b{A\norm{\tilde{u}}_{\H}+\norm f_{H^{\frac{m}{2}+q+\frac{1}{2}}\b S}},
\end{align*}
with a similar expression for $\G\tilde{u}-g$ again. Leaving $f,g$,
and $S$ fixed (or $f,g$ varying, but bounded), but varying $h_{\text{max}}$,
$\Vert\tilde{u}\Vert_{\H}$ either converges to a non-zero value, assuming
$f$ or $g$ are non-zero, or diverges to infinity due to Proposition
\ref{prop:Let--be}. With this in mind, we will often assume for simplicity
that the method given by (\ref{eq:optprob}) is $p^{\text{th}}$-order
in the sense that, 
\begin{align}
\norm{\F\tilde{u}-f}_{L^{\infty}\b S} & \le A_{p}h_{\text{max}}^{p}\norm{\tilde{u}}_{\H},\label{eq:boundsorder}\\
\norm{\G\tilde{u}-g}_{L^{\infty}\b{{\partial S}}} & \le B_{p}h_{\text{max}}^{p}\norm{\tilde{u}}_{\H}.\nonumber 
\end{align}
for some constants $A_{p},B_{p}>0$. Note that this simplification just relies on the fact that $\norm{\tilde u}_\H$ is certainly bounded below by a positive constant as long as the right-hand sides of the constraints from (\ref{eq:optprob}) are not identically zero, so some constant multiple of $\norm{\tilde u}_\H$ is greater than $\Vert f\Vert_{H^{\frac{m}{2}+q+\frac{1}{2}}\b S}$. We reiterate that $f$ here is either fixed or only allowed to vary boundedly.

We again note that (\ref{eq:boundsorder})
holds regardless of whether the PDE given by (\ref{eq:pde1}) and
(\ref{eq:pdebound}) is even solvable in $\H$; if it is, then $\norm{\tilde{u}}_{\H}$
is bounded and $\norm{\F\tilde{u}-f}_{L^{\infty}\b S},$ $\norm{\G\tilde{u}-g}_{L^{\infty}\b{{\partial S}}}\to0$.
If it is not solvable, $\norm{\tilde{u}}_{\H}$ will diverge, potentially
faster than $h_{\text{max}}^{-p}$, so $\norm{\F\tilde{u}-f}_{L^{\infty}\b S}$
and $\norm{\G\tilde{u}-g}_{L^{\infty}\b{{\partial S}}}$
could diverge.

\subsection{Hilbert Spaces for Meshfree Methods\label{subsec:Hilbert-Spaces-for}}

Recall that $\L$ must be bounded as an operator from $\H$ to $\C^{\tilde{N}+\tilde{N}_{\partial}}$
to guarantee convergence via Proposition \ref{prop:Let--be}. A straightforward
approach to construct a suitable Hilbert space is to use Fourier transforms
or series.
\begin{example}
\label{exa:Consider-the-Hilbert}Consider the Hilbert space of functions
on $\R$ given by
\[
\H:=\cb{u:u\b x=\int_{-\infty}^{\infty}e^{-\abs{\omega}+i\omega x}a\b{\omega}\,\d\omega,a\in L^{2}\b{\R}},
\]
such that if $u\b x=\int_{-\infty}^{\infty}e^{-\abs{\omega}+i\omega x}a\b{\omega}\,\d\omega$
and $v\b x=\int_{-\infty}^{\infty}e^{-\abs{\omega}+i\omega x}b\b{\omega}\,\d\omega$,
then the inner product on $\H$ is given by $\b{u,v}_{\H}:=\b{a,b}_{L^{2}\b{\R}}.$
Now consider the interpolation problem on $\{x_{j}\}_{j=1}^{\tilde{N}}\subset S\subset\R$.
\end{example}

\begin{align}
\text{minimize, over }u\in\H: & \norm u_{\H},\label{eq:rbfprob}\\
\text{subject to } & u\b{x_{j}}=f_{j}\text{ for each \ensuremath{j}\ensuremath{\in\cb{1,2,\ldots,\tilde{N}},}}\nonumber 
\end{align}
for some $\vc f\in\C^{\tilde{N}}$. Define $\L:\H\to\C^{\tilde{N}}$
by $\b{\L u}_{j}:=u\b{x_{j}}$. Then, 
\begin{align*}
\b{\L^{*}\vc{\beta}}\b x & =\sum_{j=1}^{\tilde{N}}\beta_{j}\int_{-\infty}^{\infty}e^{-2\abs{\omega}+i\omega\b{x-x_{j}}}\,\d\omega=\sum_{j=1}^{\tilde{N}}\beta_{j}\frac{4}{4+\b{x-x_{j}}^{2}}.
\end{align*}
That is, $\L^{*}\hat{\vc e}_{j}$ corresponds to the RBF $\phi$ such
that $\phi\,(x-x_{j}):=4/(4+(x-x_{j})^{2})$: an inverse quadratic
RBF. Indeed, the RBF interpolation matrix $\vc{\Phi}$ for this problem
is given by $\Phi_{jk}=4/(4+(x_{j}-x_{k})^{2})=(\L\L^{*})_{jk}.$

\begin{example}
\label{exa:Consider-placing-a}Consider placing a domain $S$ inside
a box $\Omega\subset\R^{m}$. We consider the set of periodic Fourier
basis functions on $\Omega$: $\cb{\phi_{n}}_{n=1}^{\infty}$ where
$\phi_{n}\b{\x}=e^{i\vc{\omega}_{n}\cdot\x}$ for frequencies $\vc{\omega}_{n}$
such that $\phi_{n}$ is periodic on $\Omega$. Then, define a Hilbert
space $\H$ by
\begin{equation}
\H:=\cb{\sum_{n=1}^{\infty}a_{n}d_{n}^{-\frac{1}{2}}\phi_{n}:a\in\ell^{2}},\label{eq:seriesspace}
\end{equation}
where $d=\cb{d_n}_{n=1}^\infty>0$ is a sequence that must be chosen. In the case that
$\F$ and $\G$ from Subsection \ref{subsec:Norm-Minimizing-Hermite-Birkhoff}
are linear differential operators of at most order $p$ with bounded
coefficient functions, the condition for $\L$ to be bounded (and
therefore the method to converge) is that $\norm{\vc{\omega}}_{2}^{p}d^{-\frac{1}{2}}:=\{\norm{\vc{\omega}_n}_{2}^{p}d_n^{-\frac{1}{2}}\}_{n=1}^\infty\in\ell^{2}$,
where $\{\vc{\omega}\}_{2}:=\{\Vert\vc{\omega}_{n}\Vert_{2}\}_{n=1}^{\infty}$
is a sequence of Euclidean norms of the Fourier frequencies (see \cite{venn}
for further discussion). Another possible construction of this Hilbert
space is
\[
\H:=\cb{\sum_{n=1}^{\infty}\hat{u}_{n}\phi_{n}:\sum_{n=1}^{\infty}d_{n}\abs{\hat{u}_{n}}^{2}<\infty}.
\]
In this construction, it is clear that choosing $d_{n}=(1+\Vert\vc{\omega}_{n}\Vert_{2}^{2})^{p}$
corresponds to minimum Sobolev norm interpolation (see \cite{chand18},
for example). The inner product is
\[
\b{\sum_{n=1}^{\infty}a_{n}d_{n}^{-\frac{1}{2}}\phi_{n},\sum_{n=1}^{\infty}b_{n}d_{n}^{-\frac{1}{2}}\phi_{n}}_{\H}:=\b{a,b}_{\ell^{2}}.%
\]

In dimensions greater than one, it is useful to make $d$ separable;
for example, in 2D, we could choose $d_{n}:=d_{x,n_{x}}d_{y,n_{y}}$,
where $\vc{\omega}_{n}=\omega_{x,n_{x}}\hat{\vc x}+\omega_{y,n_{y}}\hat{\vc y}$.
Then, in the case $\F,\G$ are identities (function interpolation),
it can be shown that
\begin{align}
\L^{*}\hat{\vc e}_{j} & =\sum_{n=1}^{\infty}d_{n}^{-1}e^{i\vc{\omega}_{n}\cdot\b{\x-\x_{j}}}=\b{\sum_{n_{x}=1}^{\infty}d_{x,n_{x}}^{-1}e^{i\omega_{x,n_{x}}\b{x-x_{j}}}}\b{\sum_{n_{y}=1}^{\infty}d_{y,n_{y}}^{-1}e^{i\omega_{y,n_{y}}\b{y-y_{j}}}}.\label{eq:sep}
\end{align}
This is useful because, in practice, we will always be using a truncated
Fourier series. When $d_{n}$ is separable, we can evaluate $(\L^{*}\hat{\vc e}_{j})\b{x,y}$
with $N_{\omega}=N_{\omega,x}N_{\omega,y}$ Fourier basis functions
in $\O\b{N_{\omega,x}+N_{\omega,y}}$ operations, rather than $\O\b{N_{\omega,x}N_{\omega,y}}$.
This means we can typically form the $\vc{\Phi}=\vc{VV}^*$ matrix in $\O \, (N_\omega^{1/m}(\tilde{N}+\tilde{N}_\partial)^2)$ time for a series in $m$ dimensions rather than the $\O \, (N_\omega(\tilde{N}+\tilde{N}_\partial)^2)$ of computing $\vc{VV}^*$ directly. However, the
$\vc{\Phi}$ matrix is poorly conditioned, so this approach is typically
limited to 5 or 6 digits of accuracy.

After truncating the Fourier series, the optimization problem becomes
\begin{align}
\text{minimize: } & \norm{\vc a}_{2},\label{eq:undfour}\\
\text{subject to } & \vc{Va}=\vc f,\nonumber 
\end{align}
where $\vc V$ is the matrix such that $\vc{Va}\in\C^{\tilde{N}+\tilde{N}_{\partial}}$
is a vector consisting of the values of $\F\,(\sum_{n=1}^{N_{\omega}}a_{n}d_{n}^{-1/2}\phi_{n})$
and $\G\,(\sum_{n=1}^{N_{\omega}}a_{n}d_{n}^{-1/2}\phi_{n})$
on $\{\x_{j}\}_{j=1}^{\tilde{N}}$ and $\{\vc y_{j}\}_{j=1}^{\tilde{N}_{\partial}}$,
respectively. In this case, the matrix $\vc{\Phi}$ for RBF methods
corresponds to $\vc{VV}^{*}$; one could solve $\vc{VV}^{*}\vc{\beta}=\vc f$,
then the solution to (\ref{eq:undfour}) is given by $\vc a=\vc V^{*}\vc{\beta}$.
As stated previously, this approach is poorly conditioned, however,
and superior results can generally be obtained from solving (\ref{eq:undfour})
using a singular value decomposition or complete orthogonal decomposition.
That said, if speed while using a large number of Fourier basis functions
is the main concern, forming $\vc{VV}^{*}=\vc\Phi$ using a separable $d$
as described around Eq. (\ref{eq:sep}) will be less computationally
intensive than solving (\ref{eq:undfour}) directly. Solving the positive definite system $\vc\Phi\vc\beta=\vc f$ requires $\O\,((\tilde{N}+\tilde{N}_\partial)^3)$ operations using a Cholesky decomposition or another direct algorithm, while the cost of forming a decomposition of $\vc V$ is $\O\,(N_\omega(\tilde{N}+\tilde{N}_\partial)^2)$; $N_\omega$ may be significantly larger than $\tilde{N}+\tilde{N}_\partial$ to ensure $\vc V$ is full rank. See \cite{venn} for further discussion on using these Hilbert spaces for solving surface PDEs. %

\end{example}

\section{Methods and Analysis\label{sec:Methods-and-Analysis}}

We now present two methods for estimating $\int_{S}f$ on irregular
flat domains and surfaces. Throughout this section, we assume that
$\H$ is a Hilbert space associated with a particular order of convergence
for solving a Poisson problem using a constrained optimization approach
(\ref{eq:optprob}). Specifically, we assume that there exist constants
$p,A_{p},B_{p}>0$ such that if $u\in\H$ and $\Delta_{S}u-f$ has
scattered zeros on a set of points $\{\x_{j}\}_{j=1}^{\tilde{N}}\subset S$
with fill distance $h_{\text{max}}$, then $\Vert\Delta_{S}u-f\Vert_{L^{\infty}\b S}\le A_{p}h_{\text{max}}^{p}\norm u_{\H}$
for a sufficiently smooth $f$ and small enough $h_{\text{max}}$.
Furthermore, if $\hat{\vc n}_{\partial S}\cdot\nabla_{S}u=0$ on a
set of scattered points $\{\vc y_{j}\}_{j=1}^{\tilde{N}}\subset\partial S$
with fill distance $h_{\text{max}}$, then $\norm{\hat{\vc n}_{\partial S}\cdot\nabla_{S}u}_{L^{\infty}\b{\partial S}}\le B_{p}h_{\text{max}}^{p}\norm u_{\H}$.
See the discussion leading up to (\ref{eq:boundsorder}) for an explanation
of these assumptions.

\subsection{Method 1: Poisson Solvability\label{subsec:Method-=0000231:-Poisson}}

We rely on Proposition \ref{prop:Let--be} to construct our first
method. Let $S\subset\Omega$ be the domain of integration, which
may be flat or a surface, and let $g$ be a real function so that
$\int_{S}g>0$ is known or can be easily estimated (for example, a
function with compact support contained in a flat domain $S$). Assume
we want to know $\int_{S}f$, where $f$ is a real integrable function
on $S$. Then consider the problem
\begin{align} 
\Delta_{S}u & =f-cg,\text{\ensuremath{\quad}}\hat{\vc n}_{\partial S}\cdot\nabla_{S}u\eval_{\partial S}=0,\label{eq:areapde}
\end{align}
where $\hat{\vc n}_{\partial S}$ is the conormal vector along $\partial S$, so that it is tangential to $S$ and normal to $\partial S$, and $c$ is some constant.
If $S$ is a closed surface, the boundary condition is simply removed. The problem (\ref{eq:areapde}) is only solvable when $\int_{S}\b{f-cg}=0$, or
equivalently, when
\begin{equation}
c=c_{*}:=\frac{\int_{S}f}{\int_{S}g}.\label{eq:condforc}
\end{equation}

We can then discretize problem (\ref{eq:areapde}). Let $\H$ be a
Hilbert space satisfying the assumptions at the beginning of this
section. Also assume that $f$, $g$, and $S$ are smooth enough, and $\partial S$ is such that a 
solution to ($\ref{eq:areapde}$) for the value of $c$ given
by (\ref{eq:condforc}) can be extended to some $u_{c_{*}}\in\H$. For example, in flat domains, for $f,g\in H^{p}(S)$ and a $C^{p+2}$ boundary $\partial S$, an extended solution can be found for $\H=H^{p+2}(\Omega)$ (see Ch. 4, Section 2, Thm. 4 of \cite{mikha78} for PDE regularity, and Thm. A.4 of \cite{mclea00} for the extension).
Let $\{\x_{j}\}_{j=1}^{\infty}$ be dense in $S$, and let $\{\vc y_{j}\}_{j=1}^{\infty}$
be dense in $\partial S$. Then consider the problem:
\begin{align}
\text{minimize, over }u\in\H: & \norm u_{\H},\label{eq:areadisc}\\
\text{subject to } & \Delta_{S}u\b{\x_{j}}=f\b{\x_{j}}-cg\b{\x_{j}}\text{ for each \ensuremath{j\in\cb{1,2,\ldots,\tilde{N}}}},\nonumber \\
 & \hat{\vc n}_{\partial S}\cdot\nabla_{S}u\b{\vc y_{j}}=0\text{, for each }j\in\cb{1,2,\ldots\tilde{N}_{\partial}}.\nonumber 
\end{align}
 See Subsection \ref{subsec:Solving-Surface-Poisson} for the specifics
of evaluating $\Delta_{S}u$ on surfaces. We set this up so that the
minimizing solution is $\tilde{u}=\L^{*}\vc{\beta}$, where $\L^{*}$
is as in Subsection \ref{subsec:Norm-Minimizing-Hermite-Birkhoff}
and $\vc{\Phi}:=\L\L^{*}$. Then, we can write:
\begin{align*}
\vc{\beta} & =\vc{\Phi}^{-1}\b{\vc f-c\vc g}\text{ and }\norm{\tilde{u}}_{\H}^{2}=\b{\vc f-c\vc g}^{*}\vc{\Phi}^{-1}\b{\vc f-c\vc g}.
\end{align*}
Now, this will only be bounded as $\tilde{N},\tilde{N}_{\partial}\to\infty$
when $\int_{S}\b{f-cg}=0$ using Proposition \ref{prop:Let--be}. Noticing that the expression for $\norm{\tilde u}_\H^2$ is simply an upward parabola in $c$, 
then, as $\tilde{N},\tilde{N}_{\partial}\to\infty$, the minimizer of
$\Vert\tilde u\Vert_\H^2=(\vc f-c\vc g)^{*}\vc{\Phi}^{-1}(\vc f-c\vc g)$ as a function
of $c$ must approach the value of $c$ where $\int_{S}\,(f-cg)=0$; call this value $c_*:=\int_Sf/\int_Sg$. More explicitly, $\Vert\tilde u\Vert_\H^2$ is unbounded as $\tilde{N},\tilde{N}_{\partial}\to\infty$ for $c=c_*\pm\lc$ for any $\lc>0$, but bounded at $c_*$, so the vertex of $\Vert\tilde u\Vert_\H^2$ as a function of $c$ must converge to $c_*$.
This minimum (assuming $\vc{\Phi}$ is real) is at
\begin{equation}
\tilde{c}_{\tilde{N}}:=\frac{\vc g^{*}\vc{\Phi}^{-1}\vc f}{\vc g^{*}\vc{\Phi}^{-1}\vc g}\to c_*,\label{eq:esti}
\end{equation}
where we note that $\vc{\Phi}^{-1}$ is symmetric. Similar expressions to (\ref{eq:esti}) appear in statistics literature when $\H$ is a Hilbert space of functions with zero average on $S$, such as Lemma 3 of \cite{oates16} and Section 3.2 of \cite{barp22}. We wish to work on more general surfaces where an exact basis for a space of functions that integrate to zero is not known, and therefore we use a basis of functions on a larger domain $\Omega$ instead (typically a box or all of $\R^m$). This means that a different approach was required to arrive at (\ref{eq:esti}).

From (\ref{eq:esti}), we conclude that as $\tilde{N},\tilde{N}_{\partial}\to\infty$,

\begin{equation}
\int_{S}g\frac{\vc{\vc g}^{*}\vc{\Phi}^{-1}}{\vc g^{*}\vc{\Phi}^{-1}\vc g}\vc f\to\int_{S}f.\label{eq:inteconv}
\end{equation}
So, we have a linear operator to estimate $\int_{S}f$, with real
weights as long as $\vc{\Phi}^{-1}$ is real (which is typical for
a reasonable choice of $\H$). To be specific, let $\vc w^{*}=\int_{S}g\,\vc{\vc g}^{*}\vc{\Phi}^{-1}/(\vc g^{*}\vc{\Phi}^{-1}\vc g)$,
then $\vc w^{*}\vc f\to\int_{S}f$; $\vc w^{*}$ gives integration
weights.

We may also want similar expressions to (\ref{eq:esti}) and (\ref{eq:inteconv})
in terms of inner products in $\H$. Let $v_{g}^{(\tilde{N})}$
and $v_{f}^{(\tilde{N})}$ be the solutions to 
\begin{align*}
\text{minimize, over }u\in\H: & \norm u_{\H}\\
\text{subject to } & \Delta_{S}u\b{\x_{j}}=F\b{\x_{j}}\text{ for each \ensuremath{j\in\cb{1,2,\ldots,\tilde{N}}}}\\
 & \hat{\vc n}_{\partial S}\cdot\nabla_{S}u\b{\vc y_{j}}=0\text{, for each }j\in\cb{1,2,\ldots\tilde{N}_{\partial}},
\end{align*}
where $F=g$ and $F=f$ to produce $v_{g}^{(\tilde{N})}$ and $v_{f}^{(\tilde{N})}$,
respectively. These may correspond to problems on the domain $S$
that are not actually solvable, but the discretized version will still
admit a unique solution. Then, as $\tilde{N},\tilde{N}_{\partial}\to\infty$,
\[
\tilde{c}_{\tilde{N}}=\frac{\b{v_{g}^{\b{\tilde{N}}},v_{f}^{\b{\tilde{N}}}}_{\H}}{\norm{v_{g}^{\b{\tilde{N}}}}_{\H}^{2}},\quad\int_{S}g\frac{\b{v_{g}^{\b{\tilde{N}}},v_{f}^{\b{\tilde{N}}}}_{\H}}{\norm{v_{g}^{\b{\tilde{N}}}}_{\H}^{2}}\to\int_{S}f.
\]

\subsubsection{Analysis\label{subsec:Analysis}}

We now study the convergence rate of this method; in particular, we
show that it inherits the convergence rate of the underlying meshfree
PDE method.

\begin{theorem}\label{thm:method1thm}
Assume $u_c^{(\tilde N)}$ is the solution to (\ref{eq:areadisc}) for some values of $c$ and $\tilde{N}$. Assume that $f-cg$ is sufficiently smooth and $c$ is in any bounded interval, and, for some $p>0$, $\H$ produces a $p^{\text{th}}$-order meshfree method for problem (\ref{eq:areadisc}) in the sense of (\ref{eq:boundsorder}):
\begin{align}
\abs{\Delta_{S}u_{c}^{\b{\tilde{N}}}-\b{f-cg}} & \le A_{p}h_{\text{max}}^{p}\norm{u_{c}^{\b{\tilde{N}}}}_{\H}\text{ on }S,\label{eq:b1}\\
\abs{\hat{\vc n}_{\partial S}\cdot\nabla_{S}u_{c}^{\b{\tilde{N}}}} & \le B_{p}h_{\text{max}}^{p}\norm{u_{c}^{\b{\tilde{N}}}}_{\H}\text{ on }\partial S.\label{eq:b2}
\end{align}
If an exact, extended, strong-form solution $u_{c_{*}}\in\H$
for $c_{*}=\int_{S}f/\int_{S}g$ exists for problem (\ref{eq:areapde}),
then, where $\tilde c_{\tilde N}$ is the minimizer of $\big\Vert u_{c}^{(\tilde{N})}\big\Vert_{\H}$ as a function
of $c$,
\[
\abs{\tilde{c}_{\tilde{N}}\int_{S}g-\int_{S}f}\le\tilde{A}_{p}h_{\text{max}}^{p}\norm{u_{c_{*}}}_{\H}=\O\b{h_{\text{max}}^{p}},
\]
for sufficiently small $h_{\text{max}}$.

\end{theorem}

Before the proof of this theorem, it is important to note again here that $\big\Vert u_{c}^{(\tilde{N})}\big\Vert_{\H}$
may diverge as $\tilde{N}\to\infty$ and $h_{\text{max}}\to0$ for a number of reasons; in
fact it must diverge if (\ref{eq:areapde}) is not solvable for some
choice of $c$ by Proposition \ref{prop:Let--be}. Also, the sufficient smoothness of $f-cg$ is only to allow the simplification of (\ref{eq:boundsorder}). If $f-cg$, $S$, or $\partial S$ are such that (\ref{eq:areapde}) does not admit a solution that is smooth enough to be extended to $\H$, (\ref{eq:b1}) and (\ref{eq:b2}) still hold, but result in the divergence of $\big\Vert u_{c}^{(\tilde{N})}\big\Vert_{\H}$, rather than imply the convergence to zero of the left-hand sides of (\ref{eq:b1}) and (\ref{eq:b2}).

We now prove Theorem \ref{thm:method1thm}.
\begin{proof}

Let $\abs{\partial S}:=\int_{\partial S}1$, then (\ref{eq:b2}) implies that

\begin{equation}
\abs{\int_{S}\Delta_{S}u_{c}^{\b{\tilde{N}}}}=\abs{\int_{\partial S}\hat{\vc n}_{\partial S}\cdot\nabla_{S}u_{c}^{\b{\tilde{N}}}}\le\abs{\partial S}B_{p}h_{\text{max}}^{p}\norm{u_{c}^{\b{\tilde{N}}}}_{\H}.\label{eq:b3}
\end{equation}
Using (\ref{eq:b1}), (\ref{eq:b2}), and (\ref{eq:b3}),
\begin{align*}
\abs{\int_{S}\b{f-cg}} & \le\b{\abs SA_{p}+\abs{\partial S}B_{p}}h_{\text{max}}^{p}\norm{u_{c}^{\b{\tilde{N}}}}_{\H}.
\end{align*}
Then, there exists a constant $\tilde{A}_{p}$ so that
\[
\abs{c-\frac{\int_{S}f}{\int_{S}g}}\le\abs{\int_{S}g}^{-1}\tilde{A}_{p}h_{\text{max}}^{p}\norm{u_{c}^{\b{\tilde{N}}}}_{\H}.
\]
In the case that $c=\int_{S}f/\int_{S}g$ and the PDE is solvable,
this tells us nothing. However, this holds even when the PDE is not
solvable, which gives us an estimate for the rate at which $\big\Vert u_{c}^{(\tilde{N})}\big\Vert_{\H}$
diverges:
\[
\norm{u_{c}^{\b{\tilde{N}}}}_{\H}\ge\frac{h_{\text{max}}^{-p}\abs{\int_{S}g}}{\tilde{A}_{p}}\abs{c-\frac{\int_{S}f}{\int_{S}g}}.
\]

Now assume that an exact, extended, strong-form solution $u_{c_{*}}\in\H$
for $c_{*}=\int_{S}f/\int_{S}g$ exists for problem (\ref{eq:areapde}).
In this case, $\big\Vert u_{c_*}^{(\tilde{N})}\big\Vert_{\H}\le\norm{u_{c_{*}}}_{\H}$
since $u_{c_{*}}$ is feasible for (\ref{eq:areadisc}). Recall that
$\tilde{c}_{\tilde{N}}=\b{\vc g^{*}\vc{\Phi}^{-1}\vc f}/\b{\vc g^{*}\vc{\Phi}^{-1}\vc g}$
is the minimizer of $\big\Vert u_{c_*}^{(\tilde{N})}\big\Vert_{\H}$ as a function of $c$ and that we already know that $\tilde{c}_{\tilde{N}}\to c_{*}$
from Eq. (\ref{eq:esti}). Then,
\[
\norm{u_{c_{*}}}_{\H}\ge\norm{u_{c_{*}}^{\b{\tilde{N}}}}_{\H}\ge\norm{u_{\tilde{c}_{\tilde{N}}}^{\b{\tilde{N}}}}_{\H}\ge\frac{h_{\text{max}}^{-p}\abs{\int_{S}g}}{\tilde{A}_{p}}\abs{\tilde{c}_{\tilde{N}}-c_{*}}.
\]
So, we conclude that $\tilde{c}_{\tilde{N}}\to c_{*}$ at a rate $\O\b{h_{\text{max}}^{p}}$
as $h_{\text{max}}\to0$, since
\[
\abs{\tilde{c}_{\tilde{N}}-c_{*}}\le\tilde{A}_{p}\abs{\int_{S}g}^{-1}h_{\text{max}}^{p}\norm{u_{c_{*}}}_{\H}.
\]
Noting that $\tilde{c}_{\tilde{N}}\int_{S}g$ is our estimate for
$\int_{S}f$ and that $c_{*}=\int_{S}f/\int_{S}g$, we also have
\[
\abs{\tilde{c}_{\tilde{N}}\int_{S}g-\int_{S}f}\le\tilde{A}_{p}h_{\text{max}}^{p}\norm{u_{c_{*}}}_{\H}=\O\b{h_{\text{max}}^{p}}.
\]
\end{proof}
This estimate depends only on the norm of the PDE solution for $c=c_{*}$
and tells us that convergence of the integral estimate is at the rate
at which $\Delta_{S}u_{c_{*}}^{(\tilde{N})}$ and $\hat{\vc n}_{\partial S}\cdot\nabla_{S}u$
converge when solving the PDE for $c=c_{*}$.

\subsection{Method 2: Dimension Reduction\label{subsec:Dimension-Reduction}}

An alternative approach for estimating $\int_{S}f$ is to consider
\begin{equation}
\Delta_{S}u=f,\label{eq:poisson}
\end{equation}
where $S$ is a manifold with boundary. If $S$ is without boundary,
we divide it into submanifolds with boundary; this is discussed in
Subsection \ref{subsec:Meshfree-Domain-Decomposition}. We now apply
the divergence theorem:
\begin{equation}
\int_{S}f=\int_{S}\Delta_{S}u=\int_{\partial S}\nabla_{S}u\cdot\hat{\vc n}_{\partial S},\label{eq:divthm}
\end{equation}
where $\hat{\vc n}_{\partial S}$ is again the conormal vector along $\partial S$. It is useful to note that we could use any boundary
condition for (\ref{eq:poisson}). For the meshfree methods considered
in this paper, we do not require any boundary conditions at all. According
to Proposition \ref{prop:Let--be} and Corollary \ref{cor:Where--is},
as long as (\ref{eq:poisson}) admits at least one solution in $\H$,
a suitable symmetric meshfree method will produce solutions that converge
to the norm-minimizing solution to (\ref{eq:poisson}) in $\H$. This
emphasizes a key feature of these sorts of meshfree methods: problems
that are ill-posed by virtue of admitting multiple solutions become
well-posed when the norm-minimization condition is imposed.

Once the divergence theorem is applied to obtain (\ref{eq:divthm}),
we are left with an integral on a lower-dimensional manifold. The
process can then be repeated until only a line integral remains. The
discretized problem for (\ref{eq:poisson}) is, for a point cloud
$\{\x_{j}\}_{j=1}^{\tilde{N}}\subset S$,

\begin{align}
\text{minimize, over }u\in\H: & \norm u_{\H}\label{eq:optimlaplace}\\
\text{subject to } & \Delta_{S}u\b{\x_{j}}=f\b{\x_{j}}\text{ for each \ensuremath{j\in\cb{1,2,\ldots,\tilde{N}}}}.\nonumber 
\end{align}
Let $u^{(\tilde{N})}$ be the solution to this problem. If $f,S$,
and the functions in $\H$ are sufficiently smooth, then according
to (\ref{eq:boundsorder}), we have
\begin{align}
\abs{\Delta_{S}u^{\b{\tilde{N}}}-f} & \le A_{p}h_{\text{max}}^{p}\norm{u^{\b{\tilde{N}}}}_{\H}\text{ on }S,\label{eq:b1-1}
\end{align}
for some $p,A_{p}>0$. We immediately have
\begin{align*}
\abs{\int_{S}\Delta_{S}u^{\b{\tilde{N}}}-\int_{S}f}=\abs{\int_{\partial S}\nabla_{S}u^{\b{\tilde{N}}}\cdot\hat{\vc n}_{\partial S} - \int_{S}f} & \le A_{p}\abs Sh_{\text{max}}^{p}\norm{u^{\b{\tilde{N}}}}_{\H},
\end{align*}
where $\abs S:=\int_{S}1$. This can be repeated for each dimension
until $\int_{\partial S}\nabla_{S}u^{(\tilde{N})}\cdot\hat{\vc n}_{\partial S}$
is simply a line integral.

\subsection{Line Integrals\label{subsec:Line-Integrals}}

Line integrals can be evaluated by a wide variety of methods. Curves
are much simpler to parametrize than surfaces, so a standard quadrature
scheme could be implemented. However, we could also continue to use
meshfree methods by considering the first-order problem
\begin{equation}
\nabla u\cdot\hat{\vc t}_{C}=f\text{ on }C,\label{eq:firstord}
\end{equation}
where $C$ is an open curve of interest and $\hat{\vc t}_{C}$ denotes
its unit tangent vector. If we have a closed curve, such as the boundary
of a surface, we first divide it into two open curves. Then, $\int_{C}f=u\,(\vc b)-u\,(\vc a)$,
where $\vc a$ and $\vc b$ are the endpoints of $C$. To discretize
(\ref{eq:firstord}), we use

\begin{align}
\text{minimize, over }u\in\H: & \norm u_{\H}\label{eq:fodisc}\\
\text{subject to } & \b{\nabla u\cdot\hat{\vc t}_{C}}\b{\x_{j}}=f\b{\x_{j}}\text{ for each \ensuremath{j\in\cb{1,2,\ldots,\tilde{N}}}},\nonumber 
\end{align}
where $\{\x_{j}\}_{j=1}^{\tilde{N}}$ is now a point cloud on $C$.
Once again, with suitable smoothness assumptions, it is straightforward
to show that
\[
\abs{\int_{C}f-\b{u^{\b{\tilde{N}}}\b{\vc b}-u^{\b{\tilde{N}}}\b{\vc a}}}\le A_{p}h_{\text{max}}^{p}\norm{u^{\b{\tilde{N}}}}_{\H},
\]
for some $A_{p}>0$, where $u^{(\tilde{N})}$ is the solution to
(\ref{eq:fodisc}) and $p$ depends on the smoothness of $f$, $C$, and
the functions in $\H$.

\subsection{Solving Surface Poisson Problems\label{subsec:Solving-Surface-Poisson}}

A complication of solving (\ref{eq:areadisc}) or (\ref{eq:optimlaplace})
on surfaces is that it is not immediately obvious how to compute $\Delta_{S}u$
on a surface; $u$ is a function on a domain $\Omega\supset S$, not
on the surface itself, and $\Delta u\vert_{S}\ne\Delta_{S}u$ in general.
To address this, we note (see Lemma 1 of \cite{xu03})
\begin{equation}
\Delta_{S}u=\Delta u-\kappa\hat{\vc n}_{S}\cdot\nabla u-\hat{\vc n}_{S}\cdot\b{D^{2}u}\hat{\vc n}_{S}\text{ on }S,\label{eq:lapbel}
\end{equation}
where $\hat{\vc n}_{S}$ is the unit normal vector to the surface
$S$, $D^{2}u$ is the Hessian, and $\kappa=\nabla_{S}\cdot\hat{\vc n}_{S}$
is the sum of principal curvatures, which is often referred to as
the mean curvature in surface PDE communities (the usual mean curvature
would be $\kappa/2$ for a surface in this notation). Examining (\ref{eq:lapbel}),
we note that knowing or estimating $\hat{\vc n}_{S}$ is unavoidable.
For surfaces described solely by point clouds, $\hat{\vc n}_{S}$
must be estimated; this is a common problem in point cloud processing,
and can be handled by fitting a local level set to the point cloud.
This could be done with a local version of the algorithms described
in \cite{carr01,maced11}. A local high-order parametrization could
also be computed, such as in \cite{bruno07}. An analytic expression
for $\kappa$ is readily available if $S$ is described by a piecewise
parametrization or a level set. We also note that the sign of $\hat{\vc n}_{S}$
is not needed as long as the sign of $\kappa\hat{\vc n}_{S}$ is accurate.
This is noteworthy since $\kappa\hat{\vc n}_{S}$ can be estimated
locally even when the surface is non-orientable or when a consistent
orientation is difficult to obtain.

We use either of the two optimization problems below to impose $\Delta_{S}u=f$
on the point cloud. First, for when both $\kappa$ and $\hat{\vc n}_{S}$
are known or can be computed:

\begin{align}
\text{minimize, over }u\in\H: & \norm u_{\H}\label{eq:curvlapbel}\\
\text{subject to } & \Delta u-\kappa\hat{\vc n}_{S}\cdot\nabla u-\hat{\vc n}_{S}\cdot\b{D^{2}u}\hat{\vc n}_{S}=f\text{ on }\cb{\x_{j}}_{j=1}^{\tilde{N}}.\nonumber 
\end{align}
This approach is particularly useful for implicit surfaces, where
a parametrization may not be known and meshing may be difficult, but
analytic expressions for $\kappa$ and $\hat{\vc n}_{S}$ can be obtained.
The other approach, for when only $\hat{\vc n}_{S}$ can be determined,
is given by:

\begin{align}
\text{minimize, over }u\in\H: & \norm u_{\H}\label{eq:normalderzero}\\
\text{subject to } & \Delta u-\hat{\vc n}_{S}\cdot\b{D^{2}u}\hat{\vc n}_{S}=f\text{ on }\cb{\x_{j}}_{j=1}^{\tilde{N}},\nonumber \\
 & \hat{\vc n}_{S}\cdot\nabla u=0\text{ on }\cb{\x_{j}}_{j=1}^{\tilde{N}}.\nonumber 
\end{align}
We note that it is unnecessary to impose both $\hat{\vc n}_{S}\cdot\nabla u\b{\x_{j}}=0$
and $\hat{\vc n}_{S}\cdot\b{D^{2}u}\hat{\vc n}_{S}=0$ here, although
this has been done before in other surface PDE methods \cite{piret12}
and would work here as well. However, imposing both these conditions
results in larger matrices, a longer computation time, and a smaller
space of feasible interpolants. A smaller constraint set then leads
to a larger value of $\Vert\tilde{u}\Vert_{\H}$ for the interpolant
solution, which in turn can result in a less accurate solution. See
\cite{venn} for more discussion on solving surface PDEs with this
approach. 

~

\subsection{Creating Surface Subdomains\label{subsec:Meshfree-Domain-Decomposition}}

\subsubsection{Voronoi Cells}

For the method described in Subsection \ref{subsec:Dimension-Reduction},
we require $S$ to be with boundary. If $S$ is without boundary,
we must write it as the union of subdomains with boundary: $S=\bigcup_{j=1}^{N_{S}}S_{j}$.
We would like this process to be completely automated. A very simple
process is to divide $S$ into a number of Voronoi cells. Using a
sparse point cloud $\cb{\vc y_{j}}_{j=1}^{N_{S}}$, we define
\[
S_{j}:=\cb{\x\in S:\norm{\x-\y_{j}}_{2}\le\min_{k\ne j}\norm{\x-\y_{k}}_{2}},
\]
where $\norm{\cdot}_{2}$ is simply the Euclidean norm in the embedding
space (e.g. if $S$ is a surface embedded in $\R^{3}$, then the Voronoi
cells use the Euclidean norm in $\R^{3}$, not the metric on $S$).
One motivation for using Voronoi cells is that the necessary vectors
on $\partial S_{j}$ are straightforward to compute. Specifically,
$\hat{\vc t}_{\partial S_{j}}$ is in the intersection of the tangent
spaces of $S$ and the plane $P_{jk}$ that separates $S_{j}$ from
a neighbouring Voronoi cell $S_{k}$. For surfaces embedded in $\R^{3}$,
this means $\hat{\vc t}_{\partial S_{j}}\propto\hat{\vc n}_{S}\times\hat{\vc n}_{P_{jk}}$.
The outward normal to $\partial S_{j}$ is then given by $\hat{\vc n}_{\partial S_{j}}\propto\hat{\vc t}_{\partial S_{j}}\times\hat{\vc n}_{S}$.

Point placement in $S_{j}$ can then be automated using algorithms
A.1 or A.2 from \cite{venn24b}, or any other existing approach for
generating point clouds (see \cite{suchd23} for a review). As long
as $S_{j}$ is diffeomorphic to a simply connected subset of the plane
such that $\partial S_{j}$ consists of piecewise differentiable segments
(the boundaries between $S_{j}$ and neighbouring Voronoi cells),
points can be placed on $\partial S_{j}$ by placing evenly spaced
points between the adjacent corners of $S_{j}$, then moving these
points along the plane separating $S_{j}$ from a neighbouring Voronoi
cell until they intersect $S$ (via Newton's method if $S$ is a level
set, for example). Integrals on $\partial S_{j}$ are estimated using
the method from Subsection \ref{subsec:Line-Integrals} on each piecewise
differentiable segment.

\subsubsection{Planar Boundary\label{subsec:Planar-Boundary}}

A simple approach to divide closed surfaces into two subdomains is
to use a single plane defined by $P=\{\x\in\R^{3}:\vc p\cdot\vc x-c=0\}$
for some constant $c$. In this case, as long as $P$ is chosen such that it is nowhere tangent to the surface, the two subdomains are given
by
\[
S_{\pm}=\overline{\cb{\x\in S:\text{sign}\b{\vc p\cdot\vc x-c}=\pm1}},
\]
and the boundaries $\partial S_{\pm}$ are closed planar curves.
A complication that can arise for closed surfaces that bound a non-convex volume is that
each boundary may consist of a collection of disjoint curves. We note
here that the meshfree approach is primarily advantageous for surfaces
where high-order meshing may be challenging; integrals on curves may
be computed more simply by other means. A straightforward technique
for identifying and integrating over multiple planar, closed curves
is as follows:
\begin{enumerate}
\item Generate points $\cb{\y_{j}}_{j=1}^{\tilde{N}_{\partial}}$ on $\partial S_{\pm}$.
This can be done using algorithm A.1 from \cite{venn24b} modified
to constrain points to the plane via projection.
\item For each point in $\cb{\y_{j}}_{j=1}^{\tilde{N}_{\partial}}$ , find
its nearest neighbour $\vc y_{j}^{\b 1}\in\cb{\y_{k}}_{k=1,k\ne j}^{\tilde{N}_{\partial}}$.
\item For each point in $\cb{\y_{j}}_{j=1}^{\tilde{N}_{\partial}}$, find
the nearest point $\y_{j}^{\b 2}$ such that $(\vc y_{j}-\vc y_{j}^{(1)})\cdot(\vc y_{j}-\vc y_{j}^{(2)})<0$
(so that $\vc y_{j}^{\b 1}$ and $\y_{j}^{\b 2}$ are on opposite
sides of $\y_{j}$). That is, $\y_{j}^{\b 2}:=\text{argmin}\,\{\Vert\y-\vc y_{j}\Vert_2:\y\in\{\y_{k}\}_{k=1}^{\tilde{N}_{\partial}},(\vc y_{j}-\vc y_{j}^{(1)})\cdot\b{\vc y_{j}-\y}<0\}$.
\item For each trio $\y_{j},\y_{j}^{\b 1},\y_{j}^{\b 2}$, construct the
circle or line passing through each of the three points. Parametrize
circles by the angle $\theta$ between the points on the circle, the
circle's centre, and $\y_{j}$. Parametrize lines by $s$ where $s\b{\x}=(\x-\y_{j})\cdot(\y_{j}^{(1)}-\y_{j})/\Vert\y_{j}^{\b 1}-\y_{j}\Vert_{2}$
(signed distance from $\y_{j}$).
\item Compute weights for each trio so that quadratic functions in $\theta$
or $s$ are integrated exactly on the arc/line between $\y_{j}^{\b 1}$
and $\y_{j}^{\b 2}$.
\item Each point $\y_{j}$ centres a trio and should be included in two
other trios. Set the final weight for $\y_{j}$ to be half its weight
in the trio it centres added to half the sum of its weights from the
two other trios it is included in.
\end{enumerate}
This approach approximates the curve geometry by fitting a circle
to each set of three points; approximating the curve geometry is necessary
to have an order of convergence greater than two; this
approach is third-order. Note that this method also naturally handles
multiple curves by only using nearest neighbours. If $\{\y_{j}\}_{j=1}^{\tilde{N}_{\partial}}$
is dense enough on each constituent curve of $\partial S_{\pm}$,
each trio of neighbouring points will belong to the same closed curve,
correctly approximating the geometry of $\partial S_{\pm}$.

\section{Numerical Tests\label{sec:Numerical-Tests}}

\subsection{Genus-Two Surface}

Implicit surfaces and surfaces of genus greater than zero can be challenging
to mesh, particularly at low resolutions. Existing meshfree methods
often assume that the distribution from which points are selected
is known; without a parametrization of the surface, this distribution
may be unknown, and it can be difficult to generate a uniform distribution.

For the next set of tests, we examine the surface $S$ defined as
the zero level set of
\begin{equation}
\vp\b{x,y,z}:=\frac{1}{4}\b{\b{x-1}^{2}+y^{2}}^{-1}+\frac{1}{4}\b{\b{x+1}^{2}+y^{2}}^{-1}+\frac{1}{10}x^{2}+\frac{1}{4}y^{2}+z^{2}-1.\label{eq:levset}
\end{equation}
This is a genus-two surface for which we previously computed Laplace-Beltrami
eigenvalues in \cite{venn24b}. 

\subsubsection{Average Values\label{subsec:Average-Values}}

Recall that Method 1 computes ratios of integrals: $\int_{S}f/\int_{S}g$.
This means that it is straightforward to compute average values by
setting $g=1$. To start, we use approximately uniformly generated
point clouds using version B of algorithm A.1 of \cite{venn24b}
with 25 test points per point in the surface point cloud $S_{\tilde{N}}$. From a highly refined
computation, we find that the average value of $x^{2}$ on the surface
given by the zero set of $\vp$ from (\ref{eq:levset}) is 2.45884.
We compare Method 1, using (\ref{eq:normalderzero}) to solve the
Laplace-Beltrami problem, to integrating on a triangle mesh generated
with the ball-and-pivot method \cite{berna99} as
implemented by the MATLAB Lidar Toolbox's \cite{lidar} $\code{pc2surfacemesh}$,
using linear interpolation between vertices (as in P1 finite elements).
Integrating over the triangulation using a higher-order interpolation
method would require a high-order approximation of the surface geometry
and interpolation of additional points inside or outside the triangle
(see \cite{reege16} for example, for constructing a more accurate
interpolant on a triangulation using RBFs). 

For the Hilbert space $\H$, we use a separable basis as described
at the end of Example \ref{exa:Consider-placing-a} with
\begin{equation}
d_{x,n_{x}}=\b{\exp\b{q\sqrt{2\pi/T}}+\exp\b{q\sqrt{\abs{\omega_{x,n_{x}}}}}}^{2}.\label{eq:dsep}
\end{equation}
$d_{y,n_{y}}$ and $d_{z,n_{z}}$ are defined similarly, substituting
$\omega_{y,n_{y}}$ and $\omega_{z,n_{z}}$, respectively, in
place of $\omega_{x,n_{x}}$. We use $T=12$, $\Omega=[-5,5]^{3}$
and $(2\cdot80+1)^{3}$ Fourier basis functions (treated separably
to form the $\vc{\Phi}$ matrix so that this is
computationally feasible) to generate Figure \ref{fig:Convergence-plot-for},
which shows the convergence of the meshfree method compared against
triangulation and simply taking the average value on the point cloud
(Point Cloud Average), as in a Monte Carlo method. The reference value we use to estimate the relative error is computed using a triangulation with 4.8 million vertices. 

We test a couple
of different values of $q$; our observation is that for domains roughly
of this size, a choice of $q$ between 2 and 4 works reasonably well.
Functions with more oscillations are often better approximated using
lower values of $q$ and $T$ to reduce the suppression of higher frequencies.
However, as demonstrated in Figure \ref{fig:Convergence-plot-for},
an optimal selection of $q$ is not necessary to achieve rapid convergence,
and multiple values of $q$ (or $T$) could work well. The number of Fourier basis functions should be substantially larger than the total number of interpolation conditions, and larger still to capture rapid oscillations; a more detailed discussion of the selection of the number of frequencies for solving PDEs is in Remark 2 and Proposition 9 of \cite{venn}. It should be noted that
the points are not quite uniformly generated, which explains why the
“Point Cloud Average” method fails to converge. The point cloud
with $\tilde{N}=800$ is shown in Figure \ref{fig:pc800-1} (Left).

\begin{figure}[H]
\begin{centering}
\includegraphics[scale=0.5]{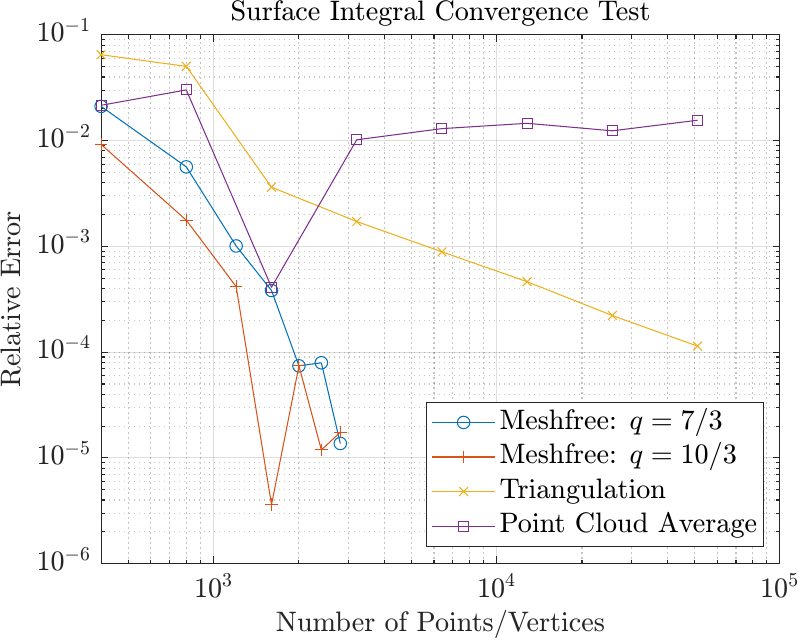}
\par\end{centering}
\caption{\label{fig:Convergence-plot-for}Convergence plot for estimates of
the average value of $x^{2}$ on a genus-two surface. Two shape parameters
are shown for the meshfree method ($q=7/3,10/3$).}
\end{figure}

There are a few observations to make regarding Figure \ref{fig:Convergence-plot-for}.
For the triangulation approach, the error is high before falling into
a pattern of consistent $\O\,(h_{\text{max}}^{2})=\O\,(\tilde{N}^{-1})$
convergence; this is due to meshing errors at low point resolutions.
The meshfree methods reach a relative error on the order of $10^{-5}$ using
only around two or three thousand points. This is roughly the error floor using
the $\vc{\Phi}$ matrix, so other strategies (namely, complete orthogonal
decomposition or singular value decomposition) would have to be employed
to solve the optimization problem with higher accuracy. For a fast and
reasonably accurate approach using a sparse point cloud, however,
the separable basis method excels. It would take a mesh with nearly
100000 vertices to reach the level of accuracy achieved by the meshfree
method with under 3000 points.

We recall that the first $\tilde{N}$ components of $\vc g^*\vc\Phi^{-1}/(\vc g^*\vc\Phi^{-1} \vc g)$ can be interpreted as integration weights for computing average values; call the vector consisting of these components $\tilde {\vc w}^*$. While not a factor in the error estimates of \ref{subsec:Analysis} as it may be in integration schemes based on exactly integrating an interpolant (see, for example, the discussion around Eq. (2.4) of \cite{vande20}), the quantity $\kappa := \sum_{j=1}^{\tilde N} |\tilde w_j|/\sum_{j=1}^{\tilde N} \tilde w_j$, which is equal to $\norm{\tilde{\vc w}}_1$ when $g=1$, is still relevant for stability. Specifically, for two functions $f_1,f_2$, our average value estimate will differ by at most $\Vert\tilde{\vc w}\Vert_1\Vert f_1-f_2\Vert_{L^\infty (S)}$. This means the method's ability to handle noise, rounding errors, and small deviations from smooth functions is determined by $\norm{\tilde{\vc w}}_1$. There is naturally a trade-off between accuracy and stability; this is demonstrated in Table \ref{tab:stability}, which shows $\Vert\tilde{\vc w}\Vert_1$ and the error in the average value of $x^2$ on $S$ for select values of $\tilde N, q,$ and $T$. We see that all values of $\Vert\tilde{\vc w}\Vert_1$ remain fairly close to one, but higher values of $T$ or $q$ generally produce more accurate results at the cost of a higher value of $\Vert\tilde{\vc w}\Vert_1$.

\begin{table}[H]
\footnotesize
\begin{centering}
\begin{tabular}{ccccccc}
\toprule 
 &  & \multicolumn{2}{c}{$\tilde N = 1600$} &  & \multicolumn{2}{c}{$\tilde N = 2400$}\tabularnewline
\midrule
$(q,T)$ &  & Relative Error & $\norm{\tilde{\vc w}}_1$ &  & Relative Error & $\norm{\tilde{\vc w}}_1$ \tabularnewline
\midrule
(7/3, 2) &  & 7.2741E-04 & 1.03 &  & 1.3862E-04 & 1.05\tabularnewline

(7/3, 12) &  & 3.8281E-04 & 1.10 &  & 7.9380E-05 & 1.14\tabularnewline

(10/3, 2) &  & 1.8272E-05 & 1.08 &  & 9.7313E-06 & 1.12\tabularnewline

(10/3, 12) &  &3.6092E-06 & 1.19 &  & 1.1916E-05 & 1.26\tabularnewline
\bottomrule
\end{tabular}
\par\end{centering}
\caption{\label{tab:stability}Relative error for computing $x^2$ on a genus-two surface and $\norm{\tilde{\vc w}}_1$ for various combinations of $\tilde N, q$, and $T$.}
\end{table}

An important note is that the meshfree approach does not require the point
cloud to be “well-spaced”. That is, the point cloud may be significantly
denser in some regions of the surface or have completely random point
spacing, and the meshfree approach will still produce a reasonable
result. Monte Carlo methods will clearly fail when points are drawn
from an unknown, non-uniform distribution, and meshing algorithms
tend to produce errors when the point cloud spacing
is extremely uneven. We use randomly generated point clouds using
algorithm A.1 of \cite{venn24b} (1 test point per point) for the
tests in Table \ref{tab:Relative-error-for} and use $q=3$ for the meshfree method, with the
other parameters left unchanged from Figure \ref{fig:Convergence-plot-for}.
Unlike in the previous test, the point cloud is not processed at all
here after points are placed on the surface; the points are not as
evenly spaced as in the previous test. This is shown
in Figure \ref{fig:pc800-1} (Centre). Meshing these point clouds
using MATLAB's ball-and-pivot routine produces many missing elements, so we
do not include it in this test.

\begin{figure}[H]
\begin{centering}
\includegraphics[scale=0.4]{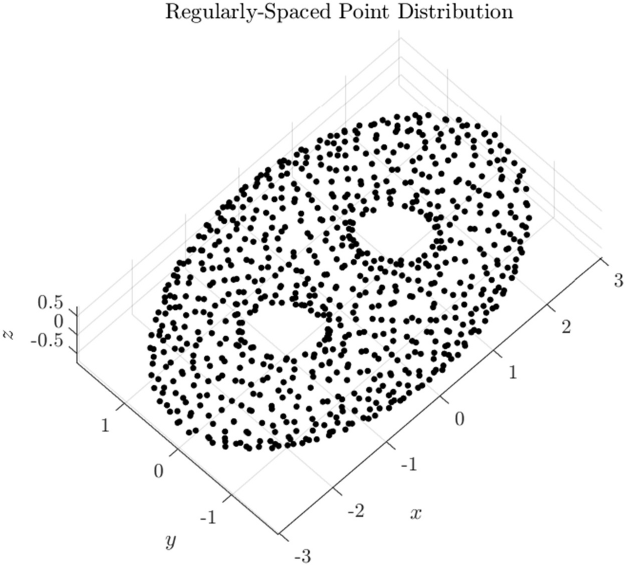}~\includegraphics[scale=0.4]{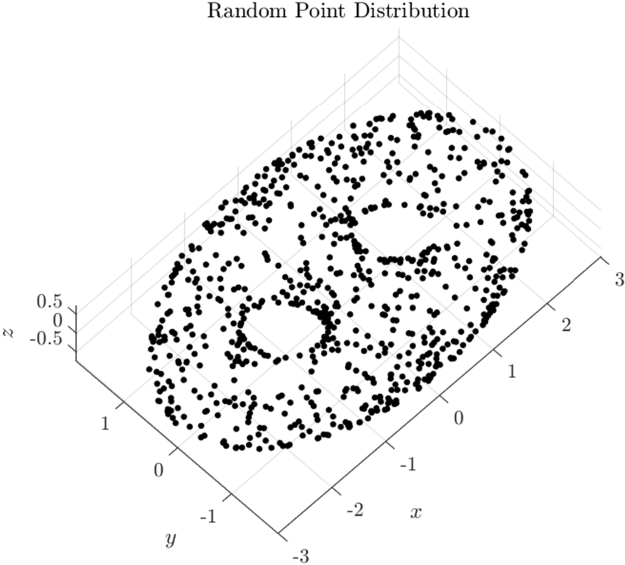}\includegraphics[scale=0.4]{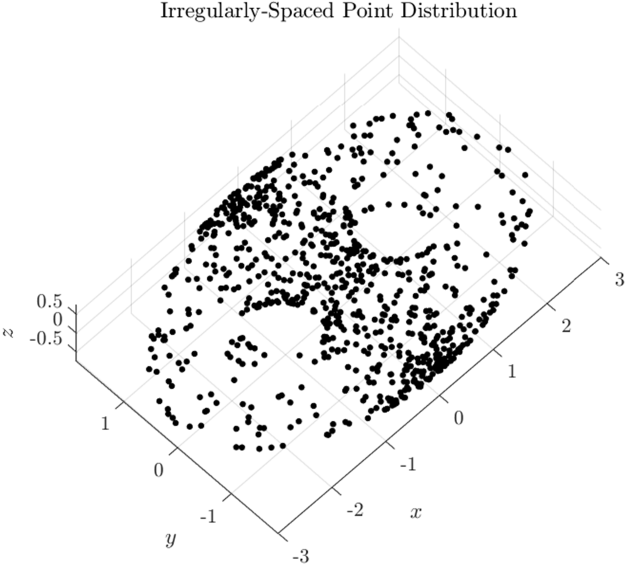}
\par\end{centering}
\caption{\label{fig:pc800-1}Three point clouds on a genus-two
surface with 800 points. Left: Regularly-spaced. Centre: Random. Right:
Irregularly-spaced.}
\end{figure}

\begin{table}[H]
\footnotesize
\begin{centering}
\par\end{centering}
\begin{centering}
\begin{tabular}{cccc}
\toprule 
$\tilde{N}$ & Relative Error (Random) &  & Relative Error (Irregular)\tabularnewline
\midrule
400 & 4.5328E-02 &  & 5.0402E-01\tabularnewline
800 & 5.6759E-03 &  & 2.0952E-01\tabularnewline
1200 & 2.4342E-03 &  & 4.3325E-02\tabularnewline
1600 & 6.1477E-04 &  & 3.5045E-02\tabularnewline
2000 & 7.0428E-04 &  & 9.5554E-03\tabularnewline
2400 & 2.3534E-05 &  & 2.0272E-03\tabularnewline
\bottomrule
\end{tabular}
\par\end{centering}
\caption{\label{tab:Relative-error-for}Left: Relative error for computing
the average value of $x^{2}$ on a genus-two surface, using randomly
generated points (Figure \ref{fig:pc800-1}, Centre) and an irregular point distribution (Figure \ref{fig:pc800-1}, Right). We use $q=3$ for both tests.}
\end{table}

The error for the meshfree method is higher than that of using a more uniform
point distribution as in Figure \ref{fig:Convergence-plot-for}, but
convergence still appears to be super-algebraic.

Finally, we produce highly non-uniform point clouds by oversampling
points with $x$ values of small magnitude. Specifically, we place
half of the points so that $x\in\b{-1,1}$, then sample the remaining
points randomly over the whole surface. An example of such a point
cloud is shown in Figure \ref{fig:pc800-1} (Right). Taking the average
value of a function on the point cloud is clearly not an option
for such a point distribution, and meshing such a point cloud will
produce a large number of errors. The same parameters as the random
point cloud test are used for Table \ref{tab:Relative-error-for}
(Right), which shows the relative error for various values of $\tilde{N}$
on the irregular point cloud. The meshfree method improves consistently
with each refinement. The errors are consistently higher here; this
is not surprising, since $h_{\text{max}}$ is significantly larger
with such an uneven point distribution.

\subsubsection{Voronoi Domain Decomposition and Surface Area\label{subsec:Domain-Decomposition-and}}

Using the procedure of Subsection \ref{subsec:Meshfree-Domain-Decomposition},
we can divide the genus-two surface into a collection of subdomains,
each diffeomorphic to a simply connected subset of the plane. An example
of such a decomposition is shown in Figure \ref{fig:Veronoi-cell-subdomains}
(Left), where both the points used for the Voronoi cells and the
surface point cloud are generated using the same procedure as the
first test in Subsection \ref{subsec:Average-Values}. 

We compare the accuracy of Method 2 against triangulation in Table
\ref{tab:Relative-error-for-1-1}, where we compute the surface area
of the genus-two surface. 100 surface patches are used for the meshfree
approach, and the Hilbert space given by (\ref{eq:dsep}) is used,
with $d_{n}$ replacing $d_{x,n_{x}}$ and $\norm{\vc{\omega}_{n}}_{2}$
replacing $\abs{\omega_{x,n_{x}}}$. The parameters $q=5,T=5$ are
chosen, where each surface patch is scaled to fit in the cube $\sb{-1/4,1/4}^{3}$,
and then the computations for each patch are performed with a Hilbert
space of functions on $\sb{-1/2,1/2}^{3}$ so that the box side length
of 1 is the same for each patch. We use $\b{2\cdot11+1}^{3}$ Fourier
basis functions in the box, and the resulting optimization
problem ((\ref{eq:normalderzero}) leading to (\ref{eq:undfour})
when the basis is truncated). The constrained problem (\ref{eq:undfour})
is solved using a complete orthogonal decomposition (as opposed to
solving $\vc{\Phi\beta}=\vc f$) for all remaining numerical tests.

\begin{table}[H]
\footnotesize
\begin{centering}
\begin{tabular}{ccccccc}
\toprule 
 &  & \multicolumn{2}{c}{Meshfree} &  & \multicolumn{2}{c}{Triangulation}\\
\midrule
$\tilde{N}$ &  & Estimate & Rel. Diff.&  & Estimate & Rel. Diff.\tabularnewline
\midrule

4000 &  & \textcolor{olive}{46}.5490998641338 & 1.4987E-03 &  & \textcolor{olive}{46}.5787702744442 & 8.6147E-04\tabularnewline

8000 &  & \textcolor{olive}{46}.\textcolor{olive}{61}32362564690 & 1.2294E-04 &  & \textcolor{olive}{46}.5985576828876 & 4.3702E-04\tabularnewline

16000 &  & \textcolor{olive}{46}.\textcolor{olive}{618}4531258073 & 1.1038E-05 &  & \textcolor{olive}{46}.\textcolor{olive}{6}088095697633 & 2.1712E-04\tabularnewline

32000 &  & \textcolor{olive}{46}.\textcolor{olive}{61896}00220989 & 1.6443E-07 &  & \textcolor{olive}{46}.\textcolor{olive}{61}37557765005 & 1.1102E-04\tabularnewline

64000 &  & 46.6189676876957 & N/A &  & \textcolor{olive}{46}.\textcolor{olive}{61}63684686587 & 5.4973E-05\tabularnewline

4800030 &  &  \ \hrulefill \   & \ \hrulefill \  &  & 46.6189312401320 & N/A\tabularnewline
\bottomrule
\end{tabular}
\par\end{centering}
\caption{\label{tab:Relative-error-for-1-1}Surface area estimates using the
meshfree method and triangulation. Digits that match all subsequent
estimates are coloured green, and the relative difference
between the estimate and each method's final estimate ($\tilde{N}=64000$ for the meshfree method, $\tilde{N}=4800030$ for the triangulation)
is given as “Rel. Diff.”.}
\end{table}

The rapid convergence of the meshfree method is apparent in this test;
we find the first 6 or 7 digits of the surface area with $\tilde{N}=64000$ while only finding
the first 4 with the triangulation. Also notable is that the meshfree
estimate with $16000$ points appears to be significantly more accurate
than the triangulation estimate with $64000$ vertices. The 4.8 million vertex triangulation is also likely less accurate than the meshfree estimates with 32000 or 64000 points.

\begin{figure}[H]
\begin{centering}
\includegraphics[scale=0.37]{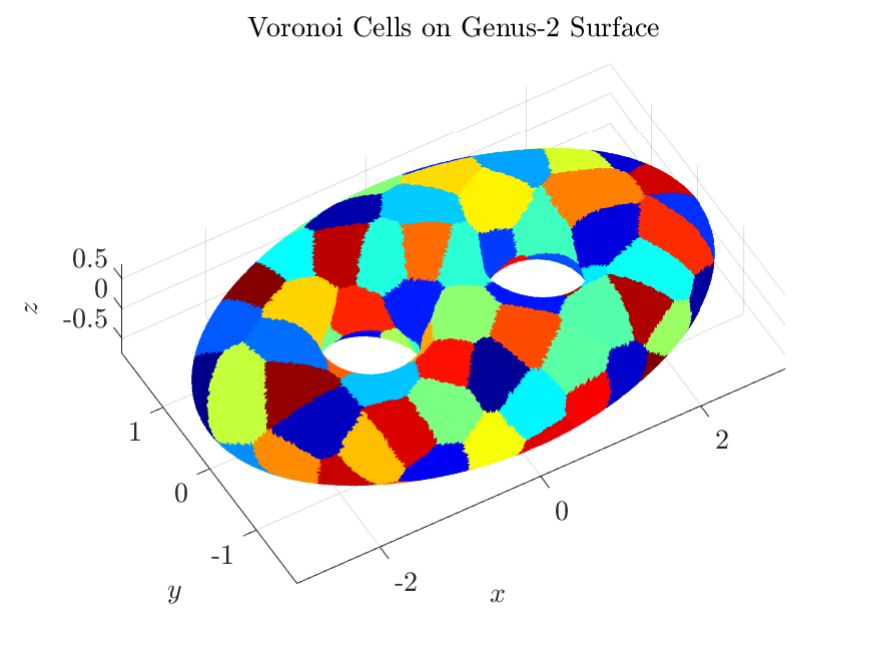}~\includegraphics[scale=0.37]{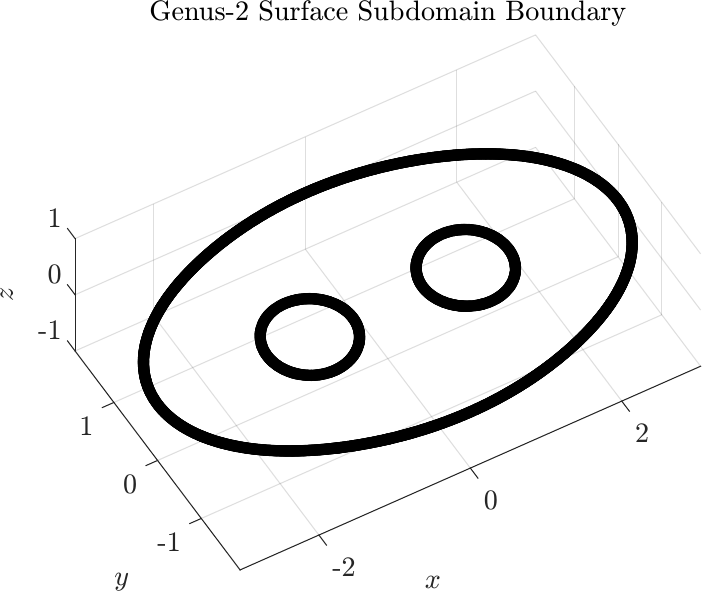}
\par\end{centering}
\caption{\label{fig:Veronoi-cell-subdomains}Left: Voronoi
cell subdomains on a genus-two surface. Right: Planar subdomain boundary
for the genus-two surface.}
\end{figure}

\subsubsection{Planar Boundary Surface Area Test\label{subsec:Planar-Boundary-Surface}}

We repeat the test from \ref{subsec:Domain-Decomposition-and} but
with two subdomains created using a planar boundary as described in
\ref{subsec:Planar-Boundary}; the boundary curves are shown in Figure
\ref{fig:Veronoi-cell-subdomains} (Right), discretized with 2000
points. The $z=0$ plane is used to separate the subdomains. The Hilbert
space given by (\ref{eq:dsep}) is used again, with $d_{n}$ replacing
$d_{x,n_{x}}$, $\Vert\vc{\omega}_{n}\Vert_{2}$ replacing $|\omega_{x,n_{x}}|$,
and $q=5,T=10$, in a box with dimensions $10\times6\times2$ with
$(2\cdot13+1)^{3}$ Fourier modes. For this test, we use (\ref{eq:curvlapbel})
to solve the Laplace-Beltrami problem. The surface area is estimated
again in Table \ref{tab:Relative-error-for-1-1-2}, with the relative
error computed using the $\tilde{N}=64000$ result from Table \ref{tab:Relative-error-for-1-1}
as an approximate “exact” value. We see that we are able to
obtain an accurate estimate with far fewer points than the previous
Voronoi test, though the Voronoi test used fewer points \textit{per
subdomain}, which is the main factor in computation time. This method
achieves a more accurate estimate using $2560$ points than a triangulation
does with 64000, and the 2560 point result is competitive with the 4.8 million vertex triangulation. We also estimate the convergence rate in Table \ref{tab:Relative-error-for-1-1-2}
and see that the method appears to converge at a high-order rate as
expected.

\begin{table}[H]
\footnotesize
\begin{centering}
\begin{tabular}{cccc}
\toprule 
$\tilde{N}$ & Estimate & Relative Error & Convergence ($\O\b{h_{\text{max}}^{p}}$)\tabularnewline
\midrule
320 & 41.9984565553679 & 9.9112E-02 & N/A\tabularnewline%
640 & 45.6178349766375 & 2.1475E-02 & 4.413\tabularnewline
1280 & 46.5745212157448 & 9.5340E-04 & 8.987\tabularnewline
2560 & 46.6180662106845 & 1.9337E-05 & 11.247\tabularnewline
5120 & 46.6189398270088 & 5.9763E-07 & 10.032\tabularnewline
\bottomrule
\end{tabular}
\par\end{centering}
\caption{\label{tab:Relative-error-for-1-1-2}Surface area estimates using
the meshfree method with a planar boundary between subdomains. Convergence
rate is relative to $h_{\text{max}}$, which scales as $\O\b{\tilde{N}^{-\frac{1}{2}}}$.}
\end{table}

\subsection{High-Order Approximation of Singular Integrals\label{subsec:High-Order-Approximation-of}}

In certain applications, particularly in the numerical approximation
of PDEs, the function to be integrated has a singularity. Crucially,
the location of the singularity is typically known. As it turns out,
meshfree methods are well suited to approximating these integrals
with a slight modification of the methods described in Subsection
\ref{subsec:Norm-Minimizing-Hermite-Birkhoff}.

Assume that we want to integrate a function $f:S\to\R$ that has a singularity
at some point $\x_{0}\in S$. As before, we could attempt to find
a function $u$ such that $\Delta_{S}u=f$ to convert $\int_{S}f$
to an integral on $\partial S$. The issue we immediately encounter
is that $u$ may also have a singularity at $\x_{0}$, which would
mean that $u$ would not be in any of the Hilbert spaces of functions commonly
used for meshfree methods, which typically consist of functions that
are at least continuous. Furthermore, we still need $\Delta_{S}v$
to be bounded on $S_{\tilde{N}}$ for any $v$ in the Hilbert space
in order to ensure boundedness of the operator $\L$ from Subsection
\ref{subsec:Norm-Minimizing-Hermite-Birkhoff}. This means that the
functions in $\H$ need to be at least $C^{2}$ at every point in
$S_{\tilde{N}}$. That is, we need to search for functions that are
possibly singular at $\x_{0}$, but are $C^{2}$ everywhere else;
such functions cannot exist in the Hilbert spaces we have used so
far.

Therefore, we need to alter our framework slightly. Rather than a
Hilbert space of functions, we consider a Hilbert space $\H$ equipped
with a (not necessarily injective) linear operator $\E$ that maps elements of the Hilbert space to a function on $\Omega$. We then consider a problem analogous to (\ref{eq:optprob}):

\begin{align}
\text{minimize, over }a\in\H: & \norm a_{\H},\label{eq:optprob-1}\\
\text{subject to}: & \b{\F\b{\E a}}\b{\x_{j}}=f\b{\x_{j}}\text{ for each \ensuremath{j\,}\ensuremath{\in\cb{1,2,\ldots,\tilde{N}},}}\nonumber \\
 & \b{\G\b{\E a}}\b{\y_{j}}=g\b{\y_{j}}\text{ for each \ensuremath{j\,}\ensuremath{\in\cb{1,2,\ldots,\tilde{N}_{\partial}}.}}\nonumber 
\end{align}

$\E$ in this case must map $a\in\H$ to a function that is $C^{p}$,
where $\F,\G$ have a maximum order of $p$, except possibly at a
finite number of points in $S$ (but not in $S_{\tilde{N}}$ or $\partial S_{\tilde{N}_{\partial}}$).
With a suitable choice of $\H$ and $\E$ it is then possible to ensure
that the constraint set for (\ref{eq:optprob-1}) is non-empty and
closed. If a solution in $a\in\H$ exists such that $\F\,(\E a)=f$
on $S$ and $\G\,(\E a)=g$ on $\partial S$, convergence is guaranteed
due to Proposition \ref{prop:Let--be}. We illustrate this approach
with a couple of examples.

\subsubsection{Singular Integrand in 2D\label{subsec:Singular-Integrand-in}}

Integrals of the form $\int_{S}f\,(\y)\,K\,(\x,\y)\,\d\y$, where $K$
has a singularity when $\x=\y$, are common when solving PDEs by boundary
integral approaches and appear in various physical
applications, particularly in electromagnetism. Let $S$ be the unit
disk and then consider the example:
\[
-\frac{1}{2\pi}\int_{S}\ln\norm{\x-\x_{0}}_{2}\,\d\x=\begin{cases}
\frac{1}{4}-\frac{\norm{\x_{0}}_{2}^{2}}{4}, & \norm{\x_{0}}_{2}\le1\\
-\frac{1}{2}\ln\norm{\x_{0}}_{2}, & \norm{\x_{0}}_{2}>1
\end{cases}.
\]
Physically, this corresponds to the electric potential inside and
outside an infinitely long cylindrical wire with constant charge density,
up to a constant; the result as a function of $\x_{0}$ solves $-\Delta u=1$
in $S$ and $0$ outside $S$. To integrate this, we search for a
function $u$ such that $\Delta u=-\frac{1}{2\pi}\ln\norm{\x-\x_{0}}_{2}.$

One possible function $u$ is $-\frac{1}{8\pi}\Vert\x-\x_{0}\Vert_{2}^{2}\ln\Vert\x-\x_{0}\Vert_{2}+\frac{1}{8\pi}\Vert\x-\x_{0}\Vert_{2}^{2}$;
we could simply use this to convert the integral into a boundary integral,
but of course the goal is to be able to integrate functions when
we do not have an expression for $u$. For example, if $S$ were instead
a curved surface and $\Delta u$ were swapped for $\Delta_S u$, we would not have such an expression for a possible $u$ available in general.

Returning to the optimization problem from (\ref{eq:optprob-1}),
for this problem, we will need to solve
\begin{align}
\text{minimize, over }a\in\H: & \norm a_{\H},\label{eq:optprob-1-1}\\
\text{subject to}: & \b{\Delta\b{\E a}}\b{\x_{j}}=-\frac{1}{2\pi}\ln\norm{\x_{j}-\x_{0}}_{2}\nonumber \\
&\text{ for each \ensuremath{j\in\cb{1,2,\ldots,\tilde{N}},}}\nonumber 
\end{align}
 When $\x_{0}\in S$, $-(2\pi)^{-1}\ln\norm{\x_{j}-\x_{0}}_{2}$
is singular at $\x_{0}$, and therefore requires our operator $\mathcal{E}$
to take elements of $\H$ to a function with a singularity at $\x_{0}$.
To be more concrete, let $\H=\ell^{2}\times\ell^{2}$, then, where
$\Omega$ is again a box that contains $S$, let $\E:\H\to C\,(\Omega)\cap C^{2}(\ensuremath{\Omega}\ensuremath{\setminus}\{\x_{0}\})$
be such that
\[
\E a=u+sv,
\]
where $u,v\in C^{q}(\Omega)$ vary with $a$ and $s\in C^{2}(\Omega\setminus\{\x_{0}\})$
is fixed and $\Delta s$ has terms that capture the expected singularity.
For example, $\Delta\,(u+sv)=-\frac{1}{2\pi}\ln\norm{\x-\x_{0}}_{2}$
is solved by $u\,(\x)=\frac{1}{8\pi}\Vert\x-\x_{0}\Vert_{2}^{2},v\,(\x)=-\frac{1}{16\pi},s\,(\x)=\Vert\x-\x_{0}\Vert_{2}^{2}\ln\,(\norm{\x-\x_{0}}_{2}^2)$.
Assuming we did not know this exact solution, but at least knew $\Delta\,(u+sv)$
needed a logarithm singularity with the scaling and any non-singular
terms undetermined, we could use $s\b{\x}=\Vert\x-\x_{0}\Vert_{2}^{2}\ln\Vert\x-\x_{0}\Vert_{2}^{2}$
and
\begin{equation}
\b{\E\b{a,b}}\b x:=\b{\sum_{n=1}^{\infty}d_{n}^{-\frac{1}{2}}a_{n}e^{i\vc{\omega}_{n}\cdot\x}}+s\b{\x}\b{\sum_{n=1}^{\infty}\tilde{d}_{n}^{-\frac{1}{2}}b_{n}e^{i\tilde{\vc{\omega}}_{n}\cdot\x}},\label{eq:E}
\end{equation}
where $\{\vc{\omega}_{n}\}$ and $\{\tilde{\vc{\omega}}_{n}\}$
correspond to Fourier frequencies on the box $\Omega$. Note that
$d_{n}$ and $\tilde{d}_{n}$ control the smoothness of $u$ and $v$,
respectively.

It is important to note that $\Delta s,\partial_{x}s,$ and $\partial_{y}s$
should not be zero on $\Omega\setminus\cb{\x_{0}}$ with a delta distribution
singularity at $\x_0$; this would allow arbitrary multiples of $s,xs$
or $ys$ to be added to our solution and change the resulting integral
(due to the delta distribution).

For our chosen $s$, a logarithm term appears in $\Delta s$, which
will allow $\ln\Vert\x-\x_{0}\Vert_{2}$ to be integrated. After computing
$u+sv$ so that $\Delta\b{u+sv} =-\frac{1}{2\pi}\ln\Vert\x-\x_{0}\Vert_{2}$
on $\tilde{N}$ scattered points, the boundary integral $\oint_{\partial S}\nabla_{}\b{u+sv}\cdot\hat{\vc n}_{\partial S}$
is computed on 1000 evenly-spaced points on the unit circle with the
standard trapezoidal rule (which converges super-algebraically for $C^{\infty}$
functions for the unit circle). The number of boundary points is chosen to ensure that the only significant component of the error is due to the surface discretization. Also note that if $\x_0$ is on or near $\partial S$, then a suitable quadrature for handling singular line integrals would be needed; a similar technique as the surface discretization or a more standard adaptive quadrature could be used.

We use $d_{n}=(\exp(q\sqrt{2\pi/T})+\exp(q\sqrt{\Vert\vc{\omega}_{n}\Vert_{2}}))^{2},q=4,T=10$,
$\Omega=[-2,2]^{2}$, and $(2\cdot30+1)^{2}$ Fourier basis functions. Interior points are generated using algorithm A.2 of \cite{venn24b} with 20 test points per point and no additional weighting of points near the boundary.
The relative error for computing $-\frac{1}{2\pi}\int_{S}\ln\Vert\x\Vert_{2}\,\d\x=\frac{1}{4}$
is shown in Figure \ref{fig:relerr1d-1-1}. High-order convergence
is observed, where we note that the point spacing varies as $h_{\text{max}}\propto\tilde{N}^{-\frac{1}{2}}$.
We note again that a unit circle is only used as a test problem to
compare to a known solution; the method's main use is for irregularly-shaped
domains and surfaces.

\begin{figure}[H]
\begin{centering}
\includegraphics[scale=0.49]{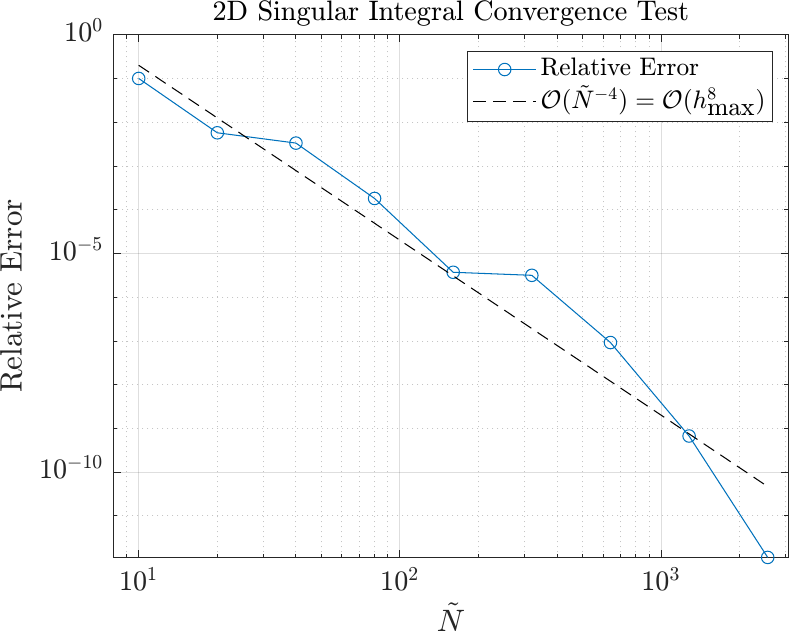}
\par\end{centering}
\caption{\label{fig:relerr1d-1-1}Relative error for computing $-\frac{1}{2\pi}\int_{S}\ln\norm{\x}_{2}\,\d\x$
on $\tilde{N}$ scattered points on the unit disk.}
\end{figure}

\subsubsection{Singular Integrand on a Surface}

We now consider the surface described by $S:\{\b{x,y,z}:\vp\b{x,y,z}=0,z\ge0\}$
where $\vp$ is the level set for a paraboloid: $\vp\b{x,y,z}:=z+x^{2}+y^{2}-1.$
We compute
\begin{equation}
\frac{1}{4\pi}\int_{S}\norm{\x-\x_{0}}_{2}^{-1}\,\d\x,\label{eq:paralog}
\end{equation}
where $\x_{0}=(0.6, 0.4, 0.48)\in S$. This integral can be interpreted as
the electric potential at $\x_{0}$ produced by a constant surface
charge on $S$. We integrate by attempting to find a function of the form $u+sv$ such that $\Delta_{S}\b{u+sv}=\frac{1}{4\pi}\Vert\x-\x_{0}\Vert_{2}^{-1}$,
where $\Delta_{S}$ is now the Laplace-Beltrami operator, and then converting the surface integral
to a line integral on the unit circle in the $xy$-plane. The line
integral is again evaluated using the trapezoidal rule and 1000 evenly
spaced points. The Laplace-Beltrami problem is discretized using a method similar to the one used in the previous test (see Eq. (\ref{eq:optprob-1-1})), but with $\Delta_S$ computed
using (\ref{eq:lapbel}) replacing $\Delta$. We seek a solution of the form $u+sv$ where
\begin{equation}\label{eq:sx}
s\b{\x}=\norm{\x-\x_{0}}_2,
\end{equation}
and $u,v\in C^{\infty}\b{\Omega}$ are unknown. This choice of $s$ gives $\Delta s=2\Vert\x-\x_{0}\Vert_2^{-1}$; however, $\Delta s$ is only one term in $\Delta_S s$. The necessary functions $u$ and $v$ to solve $\Delta_{S}\b{u+sv}=\frac{1}{4\pi}\Vert\x-\x_{0}\Vert_2^{-1}$ will vary depending on the local geometry
of the surface and will be computed numerically.  Whether such a solution will exist in the range of $\E$ is unknown, but the idea is simply to choose an $s$ term such that $\Delta_S s$ captures the singular behaviour in the Poisson problem. One can use (\ref{eq:lapbel}) to show that $\Delta_S s - \Vert\x - \x_0\Vert_2^{-1}\to 0 $ as $\x\to \x_0$ on $S$, noting that $((\x - \x_0)/\Vert\vc x -\vc x_0\Vert_2)\cdot \hat{\vc n}_S=\O\,(\Vert \x - \x_0\Vert_2)$ as $\x\to\x_0$ for $\x $ on $S$. %
We will see that including the $s$ term drastically improves the numerical results. 

The Hilbert space and $\E$ used are identical to
Subsection \ref{subsec:Singular-Integrand-in}, but with $\Omega=[-2,2]^{3}$
(instead of $[-2,2]^{2}$), $(2\cdot11+1)^{3}$ Fourier basis
functions, and $s$ given by (\ref{eq:sx}). The singularity in the integrand is placed at $\x_0=(0.6, 0.4, 0.48)\in S$. Surface points are generated using version
B of algorithm A.1 of \cite{venn24b} with 25 test points per point; notably, we do not place extra points near the singularity. This is compared to the naive approach, where a solution
$u\in C^{\infty}(\Omega)$ is sought, in Table \ref{tab:Relative-error-for-1-1-1}
(effectively searching for solutions of the form $\E\b{a,0}$ only,
where $\E$ is similar to (\ref{eq:E}), but for 3 dimensions). We compare to a value computed (0.613817759) to machine precision from a polar integral resulting from a parametrization, and
see that the augmented method converges to the first 7 or 8 digits
by $\tilde{N}=2560$, while the naive approach fails to converge beyond
a digit or two. The Poisson solution using the augmented approach
for $\tilde{N}=2560$ is shown in Figure \ref{fig:relerr1d-1-1-1}. We note that we do not increase the point density near the singularity at $\x_0$. We also point out that $\x_0$ is not at the vertex of $S$ and no boundary conditions are imposed, so the norm-minimizing solution need not be symmetric about the singularity or appear related to the boundary in any way.

\begin{table}[H]
\footnotesize
\begin{centering}
\begin{tabular}{ccccccc}
\toprule 
 &  & \multicolumn{2}{c}{Augmented} &  & \multicolumn{2}{c}{Naive}\tabularnewline
\midrule
$\tilde{N}$ &  & Estimate & Relative Error &  & Estimate & Relative Error\tabularnewline
\midrule
80 &  & 0.613306203 & 8.3340E-04 &  & 0.589499078 & 3.9619E-02\tabularnewline

160 &  & 0.613575680 & 3.9438E-04 &  & 0.589672687 & 3.9336E-02\tabularnewline

320 &  & 0.613791688 & 4.2475E-05 &  & 0.519003456 & 1.5447E-01\tabularnewline

640 &  &0.613813479 & 6.9734E-06 &  & $-$4.283902192 & 7.9791E+00\tabularnewline

1280 &  & 0.613819810 & 3.3411E-06 &  & $-$3.565659819 & 6.8090E+00\tabularnewline

2560 &  & 0.613817767 & 1.2460E-08 &  & 0.370804966 & 3.9590E-01\tabularnewline
\bottomrule
\end{tabular}
\par\end{centering}
\caption{\label{tab:Relative-error-for-1-1-1}Estimate of (\ref{eq:paralog})
using solutions augmented with a singular term (Augmented) and a naive
approach (Naive).}
\end{table}

\begin{figure}[H]
\begin{centering}
\includegraphics[scale=0.49]{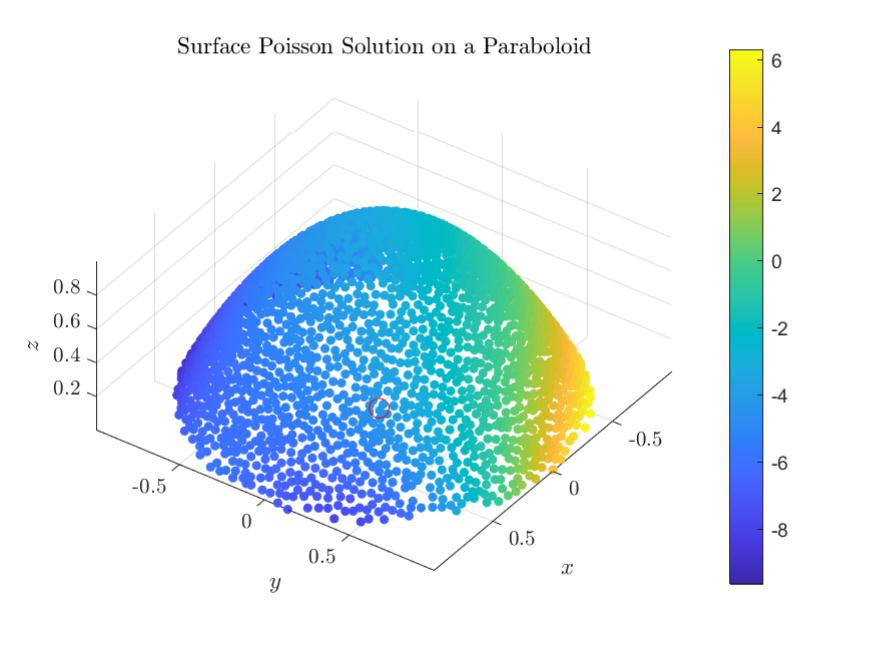}
\par\end{centering}
\caption{\label{fig:relerr1d-1-1-1}A computed solution to $\Delta_{S}u=\frac{1}{4\pi}\Vert\x-\x_{0}\Vert_{2}^{-1}$
on the paraboloid using the augmented method with $\tilde{N}=2560$
points ($\x_{0}=\b{0.6,0.4,0.48}$). The location of the singularity in the integrand is shown with a red circle.}
\end{figure}

\section{Conclusions}

We presented, analyzed, and tested two high-order, meshfree techniques
for surface integration. Method 1 (Subsection \ref{subsec:Method-=0000231:-Poisson})
uses a result relating solvability and boundedness (Proposition \ref{prop:Let--be})
to determine which scaling of a function makes a Poisson problem solvable,
then deduces the ratio of the integral of two functions on a domain
of interest. This approach is useful for finding the average values of
functions on surfaces, as shown in the tests of \ref{subsec:Average-Values}, and can be used directly on closed surfaces without splitting them into subdomains.
We also demonstrated numerically that reasonable estimates of the
average value of a function can be obtained using our approach even
on irregular, non-uniform point distributions. Method 2 (Subsection
\ref{subsec:Dimension-Reduction}) uses the divergence theorem to
reduce surface integrals to line integrals and takes advantage of
the ability of symmetric meshfree methods to find one of many possible
solutions to PDEs; this avoids imposing boundary conditions. Avoiding boundary conditions is helpful, since PDEs with boundary conditions may not permit a smooth solution if the boundary is only piecewise smooth. Unlike method 1, method 2 computes integrals directly, not as a ratio of integrals. However, method 2 requires working on a surface with boundary; for closed surfaces we introduce an artificial boundary by cutting the surface into at least two patches, as discussed in \ref{subsec:Meshfree-Domain-Decomposition}. We demonstrated method 2's accuracy
in \ref{subsec:Domain-Decomposition-and} and \ref{subsec:Planar-Boundary-Surface},
where we estimated the surface area of a genus-two surface considerably
more accurately than a standard triangulation approach. Both methods
also preserve the super-algebraic approximation properties of meshfree
interpolation methods and do not require prior knowledge of the probability
distribution from which a point cloud is sampled; this differs from
Monte Carlo methods which are $\O(\tilde{N}^{-\frac{1}{2}})$ and
typically assume a certain sampling distribution.

In Subsection \ref{subsec:High-Order-Approximation-of}, we generalized
the optimized Hermite-Birkhoff problem solved by symmetric meshfree
methods to handle singularities while maintaining high-order convergence.
We
then used the approach for solving a Poisson problem with a singular
source term to evaluate singular integrals in 2D and on a surface,
achieving rapid convergence in both cases.

Additional concepts we introduced included separable basis functions
for Fourier series meshfree methods (see the end of Example \ref{exa:Consider-placing-a}),
which allow for the rapid and systematic evaluation of matrices that
appear when using a symmetric matrix approach to solving the optimization
problem; solving $\vc{\Phi\beta}=\vc f$ directly is known as RBF-direct
in the literature when radial basis functions are used. Constructing
the correct $\vc{\Phi}$ matrices when using Hermite RBFs can be difficult,
particularly for surface PDEs, since it involves computing multiple
derivatives of the original RBF. Even for a simple Laplace-Beltrami
problem, all fourth-order derivatives of the RBF must be computed
to form the correct $\vc{\Phi}$, which can be cumbersome. Fourier
series, particularly separable ones, are straightforward to work with
by comparison and allow for greater flexibility in balancing conditioning
and convergence, since the weights $d_{n}>0$ can be chosen freely.
We also introduced simple, automated approaches for implicit surface
domain decomposition for meshfree methods (Subsection \ref{subsec:Meshfree-Domain-Decomposition}).

In future work, we seek to use the introduced integration techniques
for solving PDEs numerically (in weak form or via boundary integral
methods). This will involve developing a procedure to “stabilize”
the integration weights; high-order integration schemes, including
the ones presented here, can produce negative integration weights.
These jeopardize the stability of time-stepping methods. We have observed
that points with negative weights can be repeatedly removed to eventually
produce a positive set of weights while maintaining accuracy, but
this warrants further testing. Adaptive point selection techniques
are another area of future work, since the meshfree nature of these
approaches allows complete freedom in point placement without the
need to repeatedly re-mesh.

\bibliographystyle{siamplain}

\end{document}